\documentclass[12pt]{amsart}
\usepackage{amssymb}
\usepackage{amsmath}
\usepackage{graphicx}
\usepackage{amssymb}
\usepackage{amsthm}
\usepackage{epsf}
\usepackage{amsbsy,amsmath}
\usepackage{mathtools}
\usepackage{mathrsfs}
\usepackage{amsfonts}
\usepackage{amssymb}
\usepackage{enumitem}
\usepackage{eucal}
\usepackage{graphics,mathrsfs}
\usepackage{amsthm}
\usepackage{secdot}
\usepackage{esint}
\usepackage[colorlinks,citecolor=blue]{hyperref}
\usepackage{varwidth}
\usepackage{tasks}

\theoremstyle{plain}
\newtheorem{theorem}{Theorem}[section]

\newtheorem{lemma}[theorem]{Lemma}

\theoremstyle{definition}
\newtheorem{definition}[theorem]{Definition}

\newtheorem{remark}[theorem]{Remark}

\allowdisplaybreaks

\newcommand{\G}{\mathbb G}
\usepackage{xcolor}

\long\def\symbolfootnote[#1]#2{\begingroup
\def\thefootnote{\fnsymbol{footnote}}\footnote[#1]{#2}\endgroup}

\numberwithin{equation}{section}
\begin{document}
\title[Best constants and ground states on stratified Lie groups] { Best constants in subelliptic fractional  Sobolev and Gagliardo-Nirenberg inequalities and ground states on stratified Lie groups}
\author{Sekhar Ghosh, Vishvesh Kumar and Michael Ruzhansky}
\address[ Sekhar Ghosh]{Department of Mathematics, National Institute of Calicut, Kozhikode, Kerala, India - 673601}
\email{sekharghosh1234@gmail.com / sekharghosh@nitc.ac.in}
\address[Vishvesh Kumar]{Department of Mathematics: Analysis, Logic and Discrete Mathematics, Ghent University, Ghent, Belgium}
\email{vishveshmishra@gmail.com / vishvesh.kumar@ugent.be}

\address[Michael Ruzhansky]{Department of Mathematics: Analysis, Logic and Discrete Mathematics, Ghent University, Ghent, Belgium\newline and \newline
School of Mathematical Sciences, Queen Mary University of London, United Kingdom}
\email{michael.ruzhansky@ugent.be}


\maketitle

\begin{abstract} 
In this paper, we establish  the sharp fractional subelliptic Sobolev inequalities and  Gagliardo-Nirenberg inequalities on stratified Lie groups. The best constants are given in terms of a ground state solution of a fractional subelliptic equation involving the fractional $p$-sublaplacian ($1<p<\infty$) on stratified Lie groups. We also prove the existence of ground state (least energy) solutions to nonlinear subelliptic fractional  Schr\"odinger equation  on stratified Lie groups.  Different from the proofs of analogous results in the setting of classical Sobolev spaces on Euclidean spaces given by Weinstein (Comm. Math. Phys. 87(4):576-676 (1982/1983)) using the rearrangement inequality which is not available in stratified Lie groups, we apply a subelliptic version of vanishing lemma due to Lions extended in the setting of stratified Lie groups combining it with the compact embedding theorem for subelliptic fractional Sobolev spaces obtained in our previous paper (Math. Ann. (2023)).  We also present subelliptic fractional logarithmic Sobolev inequalities with explicit constants on stratified Lie groups. The main results are new for $p=2$ even in the context of the Heisenberg group. 
\end{abstract}

\tableofcontents

\section{Introduction and main results}
The Sobolev and  Gagliardo-Nirenberg inequalities play an important role in the study of partial differential equations (PDEs) (see \cite{BB98}).
The classical Gagliardo-Nirenberg inequality in the Euclidean space $\mathbb{R}^N$ was established by Gagliardo \cite{Gag59} and Nirenberg \cite{Nir59} in their celebrated papers independently. They proved that for every $u \in W^{1,2}( \mathbb{R}^N)$, there exists a positive constant $C$ such that the following inequality holds.
\begin{align}\label{r1.1}
	\int_{ \mathbb{R}^N}|u|^{q} dx\leq C\left(\int_{ \mathbb{R}^N}|\nabla u|^{2} d x\right)^{\frac{N(q-2)}{4}}\left(\int_{ \mathbb{R}^N}|u|^{2} d x\right)^{\frac{2 q-N(q-2)}{4}},
\end{align}
    where $2<q<2^*.$ Here, $2^*=\frac{2N}{N-2}$ for $N\geq 3$ and $2^* = +\infty$ for $N = 2$. The smallest $C$ such that the inequality \eqref{r1.1} holds is called the best constant and we denote it by $C_{GN, \mathbb{R}^N}$. Finding the best constants in the Gagliardo-Nirenberg  inequalities  has an analytical or geometrical significance and  proved to be vital in studying the well-posedness of Cauchy problems and stability theorems for nonlinear parabolic PDEs.    The pioneering work establishing a relationship between the best constant and least energy solutions is due to Weinstein \cite{Wei82}. For $N=1,$ the best constant in the Gagliardo-Nirenberg inequality was calculated by Nagy \cite{Nagy41}.  Weinstein \cite{Wei82} provided a form of $C_{GN, \mathbb{R}^N}$ given by the least energy solution of the following stationary Schr\"odinger equation.
\begin{equation}\label{r1.2}
		-\Delta u+u=|u|^{p-2}u, \quad u \in W^{1,2}(\mathbb{R}^N).
\end{equation}
 It is noteworthy to mention that the expression for the best constant does not depend on the least energy solution. Following these studies, there have been many contributions to the literature on the Euclidean case which either extended and improved these inequalities or  deal with the best constants and extremisers for such inequalities.  For instance, we mention \cite{A76,B1,B2,B3,delD02,DELL14,DT16,BG05,DHHL,B11,FW19,HMOW11,HYZ12,MP18,NPV12,T76,Esf15,DLL22,MS02,zhang} and the references therein without making an attempt to present an exhaustive list.

 In this paper, we are interested in the analysis of the best constants appearing in the fractional subelliptic Gagliardo-Nirenberg inequalities on stratified Lie groups, a subclass of nilpotent Lie groups. These groups naturally appear in analysis, representation theory, and geometry. One prominent example of such a group is the Heisenberg group.  It was noted in the celebrated paper \cite{RS76} by Rothschild and Stein that nilpotent Lie groups play an important role in deriving sharp subelliptic estimates for differential operators on manifolds. In 1970, Stein \cite{Stein71} delivered a visionary program at the ICM in Nice for studying analysis and PDEs on stratified Lie groups  (see also \cite{Cart28}).  He demonstrated it in his seminal joint paper with Rothschild on the Rothschild-Stein lifting theorem that says that a general H\"ormander's sums of squares of vector fields on manifolds can be approximated by the  sublaplacian on some stratified Lie groups (see also, \cite{F77} and \cite{Roth83}).  We also refer to Gromov \cite{Gro} or Danielli,
Garofalo and Nhieu \cite{DGN07} for general exposition from different points of view and \cite{C98, CGS04} for  applications in mathematical models of crystal material and human vision.

The functional inequalities on Lie groups, in particular, on stratified Lie groups have been extensively studied during the last few decades. We refer to papers \cite{F75,IV11,CCR15,RTY20,RT16,L19,CDG93,KD20,CR13II,GL92,AM18,KRS20,GLV,GKR23,KRS23}, monographs \cite{FS82,RS19} and references therein for an excursion into the world of subelliptic functional inequalities and their applications.  In this paper, we are concerned with the fractional subelliptic Sobolev and Gagliardo-Nirenberg inequalities on stratified Lie groups and their best constants. In fact, we will show, as in the seminal paper of Weinstein \cite{Wei82}, that the best constant in the  subelliptic fractional Gagliardo-Nirenberg inequality can be explicitly expressed as the least energy solution of the nonlocal subelliptic stationary Schr\"odinger equation on the stratified Lie group $\G$:

\begin{equation}\label{r1.71}
	\left(-\Delta_{p, \mathbb{G}}\right)^s u+|u|^{p-2} u = |u|^{q-2} u, \quad u \in W_{0}^{s,p}(\mathbb{G}),
\end{equation}

where $s \in (0, 1),$ $p \in (1, \infty)$ and $p<q<p_s^*:=\frac{Qp}{Q-ps}$ with $Q$  being the homogeneous dimension of $\G$ associated with the group dilations $(D_\lambda)_{\lambda>0}$. Here the operator $\left(-\Delta_{p, \mathbb{G}}\right)^s$ on $\G$ is the nonlinear nonlocal counterpart of the sublaplacian on a stratified Lie group $\G.$ For $p=2,$ the operator turns out to be the fractional sublaplacian $(-\Delta_\G)^s,$ which was an object of deep investigation and its connection with different areas of mathematics in many intriguing papers \cite{BF13,RT16,FMMT15,GLV,K20,RT20} and references cited therein. This is our guiding operator defined, initially for $u \in C_c^\infty(\G),$ by 
\begin{equation}\label{fracl}
(-\Delta_\G)^s(u)(x):= \lim_{\epsilon \rightarrow 0} \int_{\G \backslash B(x, \epsilon)} \frac{|u(x)-u(y)|}{|y^{-1}x|^{Q+2s}} dy = C_{Q, s}\,\, P.V. \int_{\G} \frac{|u(x)-u(y)|}{|y^{-1}x|^{Q+2s}} dy.
\end{equation}
It is known that for a $H$-type group the operator $(-\Delta_\G)^s$ is a multiple of the pseudo-differential operator defined as 
\begin{equation}
\mathcal{L}_s:=2^s (-\Delta_z)^{\frac{s}{2}} \frac{\Gamma (-\frac{1}{2}\Delta_\G (-\Delta_z)^{\frac{1}{2}}+\frac{1+s}{2}) }{\Gamma (-\frac{1}{2}\Delta_\G (-\Delta_z)^{\frac{1}{2}}+\frac{1-s}{2})},
\end{equation}
where $-\Delta_z$ is the positive Laplacian in the center of the $H$-type group $\G$ and $\Delta_\G$ is the sublaplacian on $\G.$ For more details, we refer to \cite{BF13,RT16,RT20}. It is worth noting that $\mathcal{L}_s$ is a ``conformal invariant" operator  and is very important in CR geometry (see \cite{FMMT15}). Finally, we would like to point out that the operator $(-\Delta_\G)^s$ does not coincide with the standard fractional power $-\Delta_{\G}^s$ of the sublaplacian $-\Delta_{\G}$ in the Heisenberg group or in general, $H$-type groups for any value of $s \in (0, 1),$ which is defined as 

$$(- \Delta_{\G}^su)(x):=-\frac{s}{\Gamma (1-s)}\int_0^\infty \frac{1}{t^{1+s}}(H_tu(x)-u(x))\,dt, $$
where $H_t:=e^{-t \Delta_\G}$ is the heat semigroup constructed by Folland \cite{F75}. 

On the other hand, many relevant results have been achieved during the last few years concerning the nonlinear subelliptic equations on stratified Lie groups, we cite \cite{FF15,FBR17,GLV,BFG,Gas20,Gas21,GKR23,MPPP23,L15,L19,PP22,PT21,CLZ23} and reference therein just to mention a few of them. Motivated by the representation \eqref{fracl} of the nonlocal operator $(-\Delta_{\G})^s,$ in this paper we will work with a more general nonlocal nonlinear operator involving the $p$-growth of the norm, known as, the fractional $p$-sublaplacian on $\G$ and defined, for $s\in(0,1)$ and $p\in[1,\infty)$,  by 
\begin{equation}
	\left(-\Delta_{p,{\mathbb{G}}}\right)^s u(x):=C_{Q,s, p}\,\,  P.V. \int_{{\mathbb{G}} } \frac{|u(x)-u(y)|^{p-2}(u(x)-u(y))}{\left|y^{-1} x\right|^{Q+p s}} dy, \quad x \in {\mathbb{G}},
\end{equation}
for $u \in C_c^\infty(\G).$ 
To state our first result regarding the existence of the least energy solution of \eqref{r1.71} involving the fractional $p$-sublaplacian on $\G$, we first recall some basic definitions.
Let $\Omega \subset {\mathbb{G}}$ be an open subset. Then for $0<s<1\leq  p<\infty$, the fractional Sobolev space $W^{s,p}(\Omega)$ on stratified groups is defined as 

\begin{equation}
	W^{s,p}(\Omega)=\{u\in L^{p}(\Omega): [u]_{s, p,\Omega}<\infty\},
\end{equation}
endowed with the norm 
\begin{equation}
	\|u\|_{W^{s,p}(\Omega)}^p:=\|u\|_{L^p(\Omega)}^p+[u]_{s,p,\Omega}^p,
\end{equation}
where $[u]_{s, p,\Omega}$ denotes the Gagliardo semi-norm defined by
\begin{equation}
	[u]_{s, p,\Omega}:=\left(\int_{\Omega} \int_{\Omega} \frac{|u(x)-u(y)|^{p}}{\left|y^{-1} x\right|^{Q+ps}} dxdy\right)^{\frac{1}{p}}<\infty.
\end{equation}

Observe that for all $\phi \in C_c^{\infty}(\Omega)$, we have $[u]_{s, p,\Omega}<\infty$. We define the space  $W_0^{s,p}(\Omega)$ as the closure of $C_c^{\infty}(\Omega)$ with respect to the  norm $\|u\|_{W^{s,p}(\Omega)}$. We would like to point out that $W_0^{s,p}(\mathbb{G})=W^{s,p}(\mathbb{G})$ (see \cite{GKR23}). 

\noindent  With the definitions of $I$ and $\mathcal{N}$ as in Section \ref{sec3}, the following theorem is our first main result.
\begin{theorem}\label{rt1.2i}
	Let $\G$ be a stratified Lie group with homogeneous dimension $Q.$ Let $0<s<1< p<\infty$ and $p<q<p_s^*:=\frac{Qp}{Q-ps}$. Then the problem \eqref{r1.71} has a least energy solution $\phi \in W_{0}^{s,p}( \mathbb{G})$. Moreover, we have the least energy $d=I(\phi):=\inf_{u\in\mathcal{N}}I(u)$. 
\end{theorem}

 This result in the classical case, that is, for  \eqref{r1.2} was proved by Weinstein \cite{Wei82} and using this he obtained the sharp estimates of the best constant $C_{GN, \mathbb{R}^N}$ in the Gagliardo-Nirenberg inequality \eqref{r1.1} by solving a minimisation problem. The key ingredients of his proof were symmetric rearrangement and the compact embedding of the radial Sobolev space $W_r^{1,2}(\mathbb{R}^N)$ in $L^p(\mathbb{R}^N)$ for $N \geq 2$ and $2<q<2^*.$ It is now well-known that symmetric rearrangement techniques are not available for the stratified Lie groups and therefore one needs to seek a different proof. Chen and Rocha \cite{CR13II} extended the results of Weinstein \cite{Wei82} in the Heisenberg group by using the classical compact embedding of the Folland-Stein-Sobolev space (\cite{GL92}) and Lions type vanishing lemma (\cite{Li85I,Li85II}) for respective Folland-Stein-Sobolev space on the Heisenberg group. Based on this existence result, they (\cite{CR13II}) obtained the expression for the best constant of the following Gagliardo-Nirenberg inequality on the Heisenberg group: \begin{equation}\label{GN_CR}
\int_{\mathbb{H}^{N}}|u|^{q}dx\leq C\left(\int_{\mathbb{H}^{N}}|\nabla_{H} u|^{2}dx\right)^{\frac{Q(q-2)}{4}}\left(\int_{\mathbb{H}^{N}}|u|^{2}dx\right)^{\frac{2q-Q(q-2)}{4}}.
\end{equation}
Here, $Q=2N+2$ refers to the homogeneous dimension of the Heisenberg group $\mathbb{H}^{N}$, $\nabla_{H}$ is horizontal gradient, $q\in(2, 2^*)$, where $2^*=\frac{2N}{N-2}$.
The best constant in \eqref{GN_CR}, denoted by $C_{GN, \mathbb{H}^N},$ is expressed in terms of the least energy (or ground state) solution of the subelliptic partial differential equation. \begin{equation}\label{r-prob}
-\triangle_{H}u+u=|u|^{q-2}u, \;\;u\in W^{1,2}(\mathbb{H}^{N}),
\end{equation}
where $\triangle_{H}$ is the sublaplacian on $\mathbb{H}^{N}$, and $W^{1,2}(\mathbb{H}^{N})$ is the Sobolev space on $\mathbb{H}^{N}$ with the norm
$$\|u\|_{W^{1,2}(\mathbb{H}^{N})}:=\left(\int_{\mathbb{H}^{N}}(|\nabla_{H}u|^{2}+|u|^{2})dx\right)^{1/2}.$$

Recently, the third author and his collaborators \cite{RTY20} characterised the best constant of the following Gagliardo-Nirenberg inequality over a graded Lie group $\mathbb{G}$:

		\begin{equation}\label{GN-RTY21}
			\int_{\mathbb{G}}|u(x)|^{q}dx\leq C \left(\int_{\mathbb{G}}|\mathcal{R}_{1}^{\frac{a_{1}}{\nu_{1}}}u(x)|^{p}dx\right)^{\frac{Q(q-p)-a_{2}pq}{(a_{1}-a_{2})p^{2}}}
			\left(\int_{\mathbb{G}}|\mathcal{R}_{2}^{\frac{a_{2}}{\nu_{2}}}u(x)|^{p}dx\right)^{\frac{a_{1}pq-Q(q-p)}{(a_{1}-a_{2})p^{2}}}
		\end{equation}
		for all $u\in {W}^{a_1,p}(\mathbb{G})\cap {W}^{a_2,p}(\mathbb{G})$. Here $a_{1}> a_{2}\geq0$, $1<p<\frac{Q}{a_{1}},$ $\frac{pQ}{Q-a_{2}p}\leq q\leq\frac{pQ}{Q-a_{1}p}$ and,  $\mathcal{R}_{1}$ and $\mathcal{R}_{2}$ are positive Rockland operators of homogeneous degrees $\nu_{1}$ and $\nu_{2}$, respectively. They proved that the best constant in \eqref{GN-RTY21} can be identified by the least energy solutions of the following higher order nonlinear hypoelliptic Schr\"{o}dinger equation with the power nonlinearities:
	\begin{equation}\label{RTY21-Schrodinger}
		\mathcal{R}_{1}^{\frac{a_{1}}{\nu_{1}}}(|\mathcal{R}_{1}^{\frac{a_{1}}{\nu_{1}}}u|^{p-2}\mathcal{R}_{1}^{\frac{a_{1}}{\nu_{1}}}u)+
		\mathcal{R}_{2}^{\frac{a_{2}}{\nu_{2}}}(|\mathcal{R}_{2}^{\frac{a_{2}}{\nu_{2}}}u|^{p-2}\mathcal{R}_{2}^{\frac{a_{2}}{\nu_{2}}}u)=|u|^{q-2}u.
	\end{equation} 

In this paper, we continue the aforementioned studies  \cite{CR13II, RTY20} for a more general nonlocal operator, namely, the fractional $p$-sublaplacian on the stratified Lie groups. In \cite{GKR23}, we have proved a compact embedding result for the fractional Sobolev spaces  on the stratified Lie groups (see Theorem \ref{l-3} in the next section), which is a nonlocal version of the embedding result for the Folland-Stein-Sobolev space obtained in \cite{GL92}. This is one of the main tools in the proof of Theorem \ref{rt1.2i}. The second main tool as in the aforementioned works \cite{CR13II, RTY20} is a subelliptic version of the Lions vanishing Lemma (\cite{Li85I,Li85II}) for $W_0^{s, p}(\G),$ which is stated below and will be proved in Section \ref{sec3}. 
\begin{lemma} 
    Let $s \in (0, 1), p \in (1, \infty),$ and $q$ be such that  $p\leq q<p_s^*=\frac{Qp}{Q-ps}.$ Let $(u_k)$ be a bounded sequence in $W^{s,p}_0(\G)$ with the property 
    \begin{equation} 
        \liminf_{k \rightarrow \infty} \sup_{x \in \G} \int_{B(x, 1)} |u_k|^q dy =0.
    \end{equation}
     Then there exists a subsequence, also denoted by $(u_k),$ such that $u_k \rightarrow 0$ in $L^t(\G)$ for $t \in (q, p_s^*).$
\end{lemma}

We now turn our attention to deriving the fractional Gagliardo-Nirenberg inequality. We mentioned that the proof is contained in \cite{KRS23} but the proof is quite elementary using H\"older's inequality so we include it here. Recall that the following fractional subelliptic Sobolev inequality 

\begin{equation}\label{sobo-ineq}
    \int_{\mathbb{G}}|u(x)|^{p_s^*} dx\leq C\left[\int_{\mathbb{G}}\int_{\mathbb{G}}\frac{|u(x)-u(y)|^{p}}{\left|y^{-1} x\right|^{Q+ps}} dxdy\right]^{\frac{p_s^*}{p}}
\end{equation}
for $u \in W^{s,p}_0(\G)$ was proved in \cite[Theorem 2]{KD20}. 
Now, we derive, from the H\"older inequality, that

\begin{align*}
	\int_{ \mathbb{G}}|u|^q dx=\int_{ \mathbb{G}}|u|^{ql}|u|^{q(1-l)} dx \leq \left(\int_{\mathbb{G}}|u|^{p_s^*} dx\right)^{\frac{ql}{p_s^*}}\left(\int_{ \mathbb{G}}|u|^{p} dx\right)^{\frac{q(1-l)}{p}},
\end{align*}
where $\frac{ql}{p_s^*}+\frac{q(1-l)}{p}=1$. Thus, $l=\frac{(q-p)Q}{spq}$. Therefore, eliminating $l$, we obtain

\begin{align*}
	\int_{ \mathbb{G}}|u|^q dx\leq \left(\int_{\mathbb{G}}|u|^{p_s^*} dx\right)^{\frac{(Q-ps)(q-p)}{sp^2}}\left(\int_{ \mathbb{G}}|u|^{p} dx\right)^{\frac{spq-Q(q-p)}{sp^2}}.
\end{align*}
Using the inequality \eqref{sobo-ineq}, we get the following subelliptic fractional Gagliardo-Nirenberg inequality:

\begin{align}\label{r1.6}
	\int_{ \mathbb{G}}|u(x)|^q dx&  \leq C\left(\int_{\mathbb{G}}\int_{\mathbb{G}}\frac{|u(x)-u(y)|^{p}}{\left|y^{-1} x\right|^{Q+ps}} dxdy\right)^{\frac{Q(q-p)}{sp^2}}\left(\int_{ \mathbb{G}}|u|^{p} dx\right)^{\frac{spq-Q(q-p)}{sp^2}}.
\end{align}

The best constant for the inequality \eqref{r1.6} is the smallest positive constant $C_{GN,\mathbb{G}}>0$ such that the inequality \eqref{r1.6} is true. Thus, we can define $C_{GN,\mathbb{G}}$ as 

\begin{equation}\label{GN-intro}
    C_{GN,\mathbb{G}}^{-1}:=\inf_{u\in W_0^{s,p}(\mathbb{G})\setminus\{0\}}\frac{\left(\int_{\mathbb{G}}\int_{\mathbb{G}}\frac{|u(x)-u(y)|^{p}}{\left|y^{-1} x\right|^{Q+ps}} dxdy\right)^{\frac{Q(q-p)}{sp^2}}\left(\int_{\mathbb{G}}|u(x)|^{p}dx\right)^{\frac{spq-Q(q-p)}{sp^2}}}
		{\int_{\mathbb{G}}|u(x)|^{q}dx}.
\end{equation}
Now, we state the following theorem which says that the best constants of the fractional Gagliardo-Nirenberg inequality can be characterised by the least energy solutions of the nonlinear subelliptic fractional Schr\"odinger equation \eqref{r1.71}.

\begin{theorem}\label{rt1.3}
		Let $0<s<1<p<\infty$ and $Q>ps$ and $p<q<p_s^*:=\frac{Qp}{Q-ps}$, where $Q$ is the homogeneous dimension of a stratified Lie group $\mathbb{G}$. Let $\phi$ be a least energy solution of \eqref{r1.71}. Then the smallest positive constant $C_{GN,\mathbb{G}}$ of the Gagliardo-Nirenberg inequality \eqref{GN-intro} can be characterized by
	
	\begin{align}\label{r1.10}
			C_{GN,\mathbb{G}}^{-1} &=\frac{pqs-Q(q-p)}{pqs}\left(\frac{Q(q-p)}{pqs-Q(q-p)}\right)^{\frac{Q(q-p)}{sp^2}}\|\phi\|_{L^{p}(\mathbb{G})}^{q-p}\nonumber\\
			&=\frac{pqs-Q(q-p)}{pqs}\left(\frac{Q(q-p)}{pqs-Q(q-p)}\right)^{\frac{Q(q-p)}{sp^2}}\left(\frac{pqs-Q(q-p)}{(q-p)s} d\right)^{\frac{q-p}{p}}
	\end{align}
	where $d=I(\phi):=\inf_{u\in\mathcal{N}}I(u)$, $\phi$ being a least energy solution of \eqref{r1.71}. 
\end{theorem}

We also obtain a similar result for the fractional subelliptic Sobolev inequality on stratified Lie groups. Let us now recall the following Sobolev inequality from \cite[Theorem 1.1]{GKR23}
\begin{equation}\label{ineq-Sobolev}
		\left(\int_{\mathbb{G}}|u(x)|^{q}dx\right)^{\frac{p}{q}}\leq C\left[\int_{\mathbb{G}}\int_{\mathbb{G}}\frac{|u(x)-u(y)|^{p}}{\left|y^{-1} x\right|^{Q+ps}} dxdy+\int_{\mathbb{G}}|u(x)|^{p}dx\right],
	\end{equation}
 for $u \in W^{s,p}_0(\G).$
 
The following result expresses the best constant in terms of the least energy solution of \eqref{r1.71}. 

	\begin{theorem}\label{main thm2}
		Let $0<s<1<p<\infty$ and $Q>ps$ and $p<q<p_s^*:=\frac{Qp}{Q-ps}$, where $Q$ is the homogeneous dimension of a stratified Lie group $\mathbb{G}$. Let  $C_{S, \mathbb{G}}>0$ be the least positive constant such that the inequality \eqref{ineq-Sobolev} is true. Then we have
		
		\begin{align}
		C_{S,\mathbb{G}}^{-1}&=\left(\frac{spq}{spq-Q(q-p)}\int_{\mathbb{G}}|\phi(x)|^{p}dx\right)^{\frac{q-p}{q}}\label{Sobo-const}\\
			&=\left(\frac{pq}{q-p} d \right)^{\frac{q-p}{q}},\label{Sobo-const2}
		\end{align}
		where $d=I(\phi):=\inf_{u\in\mathcal{N}}I(u)$, $\phi$ being a least energy solution of \eqref{r1.71} and 
		
		\begin{equation}
		C_{S, \mathbb{G}}^{-1}:=\inf_{u\in W_0^{s,p}(\mathbb{G})\setminus\{0\}}\frac{\int_{\mathbb{G}}\int_{\mathbb{G}}\frac{|u(x)-u(y)|^{p}}{\left|y^{-1} x\right|^{Q+ps}} dxdy+\int_{\mathbb{G}}|u(x)|^{p}dx}
		{\left(\int_{\mathbb{G}}|u(x)|^{q}dx\right)^{\frac{p}{q}}}.
	\end{equation}
	\end{theorem}

\begin{remark}\label{rr1.4}
	We do not know if $\phi$ is a unique least energy solution of \eqref{r1.71}. But equality in \eqref{Sobo-const2} and \eqref{r1.10} implies that $C_{S,\mathbb{G}}$ and $C_{GN,\mathbb{G}}$ are independent of the choice of $\phi$. Further, we want to mention here that if we replace $Q$ with the Euclidean dimension $N$, and $\mathbb{G}$ by $\mathbb{R}^N$ then both of these constants reduce to the classical results. For instance, for $\mathbb{G}=\mathbb{R}^N$, $s=1$, $p=2,$ it is the same as the result due to Weinstein \cite{Wei82} whereas for $\mathbb{G}=\mathbb{H}^N$, $s=1$, $p=2,$ it coincides with the results due to Chen and Rocha \cite{CR13II}
\end{remark}
\begin{remark}  We can easily establish a relation between the sharp constant $C_{GN, \G}$ in the fractional subelliptic Gagliardo-Nirenberg inequality \eqref{GN-intro} and the sharp constant $C_{S, \G}$ in the fractional subelliptic Sobolev inequality \eqref{ineq-Sobolev} on the stratified Lie groups. Indeed, from Theorem \ref{rt1.3} we have 
\begin{align*}
    C_{GN,\mathbb{G}}^{-1} &=\frac{pqs-Q(q-p)}{pqs}\left(\frac{Q(q-p)}{pqs-Q(q-p)}\right)^{\frac{Q(q-p)}{sp^2}}\left(\frac{pqs-Q(q-p)}{(q-p)s} d\right)^{\frac{q-p}{p}}, 
\end{align*}
which implies that 

$$C^{-\frac{p}{q}}_{GN, \G}=\frac{pqs-Q(q-p)}{pqs} \left(\frac{Q(q-p)}{pqs-Q(q-p)}\right)^{\frac{Q(q-p)}{pqs}} C^{-1}_{S, \G}.$$
\end{remark}

Now,  we  derive the logarithmic fractional Sobolev  inequalities with an explicit constant with the help of Theorem \ref{rt1.3} and Theorem \ref{main thm2}. In this regard, our result is the following theorem.

\begin{theorem} \label{loginhom} Let $s \in (0, 1),$ $p \in (1, \infty), Q>ps$ and $p<q<p^*:=\frac{Qp}{Q-ps},$ where $Q$ is the homogeneous dimension of a stratified Lie group $\mathbb{G}.$
 Then for any $u \neq 0,$ we have 
\begin{align} 
    \int_{\G} \frac{|u(x)|^p}{\|u\|_{L^p(\G)}^p} \log \left( \frac{|u(x)|^p}{\|u\|^p_{L^p(\G)}} \right) dx \leq \frac{Q}{s} \log \left( C_{S, \G, p}\frac{\left[\int_{\mathbb{G}}\int_{\mathbb{G}}\frac{|u(x)-u(y)|^{p}}{\left|y^{-1} x\right|^{Q+ps}} dxdy+\int_{\mathbb{G}}|u(x)|^{p})dx\right]}{\|u\|_{L^p(G)}^p} \right),
\end{align} where $C_{S, \G, p}$ is  given by 
$$C_{S, \G, p}:=\left( \frac{s}{Qd}\right)^{\frac{sp}{Q}}.$$
Here $d=I(\phi):=\inf_{u\in\mathcal{N}}I(u)$, $\phi$ being a least energy solution of \eqref{r1.71}.
\end{theorem}

Before stating our next result, let us discuss the idea of the proof of Theorem \ref{loginhom} and the next theorem. The method is quite simple and  uses the logarithmic H\"older inequalities (see Theorem \ref{LogHol}) combined with the subelliptic fractional Gagliardo-Nirenberg and Sobolev inequalities. For the Euclidean case, this method was possibly first used by Merker \cite{Mer08} and later in many papers in several different scenarios, for example, \cite{HYZ12} for the logarithmic Sobolev and Gagliardo-Nirenberg inequalities under the Lorentz norms and \cite{KRS20, CKR21b} for logarithmic Sobolev and Gagliardo-Nirenberg inequalities on Lie groups. We now state our last result concerning a different version of the logarithmic fractional Sobolev inequality with a homogeneous norm. This result was first proved in  \cite{KRS23} without an explicit constant.     

\begin{theorem} \label{th1.8} Let $s \in (0, 1),$ $p \in (1, \infty), Q>ps$ and $p<q<p^*:=\frac{Qp}{Q-ps},$ where $Q$ is the homogeneous dimension of a stratified Lie group $\mathbb{G}.$
 Then for any $u \neq 0,$ we have 
\begin{align} 
    \int_{\G} \frac{|u(x)|^p}{\|u\|_{L^p(\G)}^p} \log \left( \frac{|u(x)|^p}{\|u\|^p_{L^p(\G)}} \right) dx \leq \frac{Q}{s} \log \left( C_{GN, \G}^{\frac{ps}{Q(q-p)}}\frac{\left(\int_{\mathbb{G}}\int_{\mathbb{G}}\frac{|u(x)-u(y)|^{p}}{\left|y^{-1} x\right|^{Q+ps}} dxdy\right)^{\frac{1}{p}}}{\|u\|_{L^p(G)}} \right),
\end{align} where $C_{GN, \G}$ is the best constant in the fractional subelliptic Gagliardo-Nirenberg inequality given by 

$$C_{GN, \G}:=\frac{pqs}{pqs-Q(q-p)}\left(\frac{Q(q-p)}{pqs-Q(q-p)}\right)^{\frac{Q(p-q)}{sp^2}}\left(\frac{pqs-Q(q-p)}{(q-p)s} d\right)^{\frac{p-q}{p}}.$$
\end{theorem}
Here $d=I(\phi):=\inf_{u\in\mathcal{N}}I(u)$, $\phi$ being a least energy solution of \eqref{r1.71}.

The rest of the paper is organised as follows: in Section \ref{s2} we present some basics of analysis on stratified Lie groups and fractional subelliptic Sobolev spaces. Section \ref{secc3} will be devoted to the proof of Lions vanishing lemma for stratified Lie groups. The proof of Theorem \ref{rt1.2i} is contained in Section \ref{sec3}. The proof of Theorem \ref{rt1.3} regarding the best constant of the fractional Gagliardo-Nirenberg inequality on stratified Lie groups is given in Section \ref{sec4}. The proof of Theorem \ref{main thm2} is presented in Section \ref{sec5}. We end this paper in Section \ref{sec6} by discussing the proof of Theorem \ref{loginhom} and Theorem \ref{th1.8} talking about the logarithmic version of fractional subelliptic Sobolev inequalities with explicit constants. 

\section{Preliminaries: Stratified Lie groups and fractional Sobolev spaces}\label{s2}

Prior to proceeding with the main results, we recall some preliminary tools of the stratified Lie groups  and the fractional Sobolev spaces defined on them.  There are many ways to introduce the notion of stratified Lie groups, for instance, one may refer to  books and monographs \cite{FS82,BLU07,FR16,RS19}. In his seminal paper \cite{F75}, Folland extensively investigated the  properties of function spaces on these groups.  For precise studies and properties on stratified Lie group, we refer \cite{GL92, F75, FS82, BLU07, FR16,GKR23}.

\begin{definition}\label{d1}
	A Lie group $\mathbb{G}$ (on $\mathbb{R}^N$) is said to be homogeneous if, for each $\lambda>0$, there exists an automorphism $D_{\lambda}:\mathbb{G} \rightarrow\mathbb{G}$  defined by $D_{\lambda}(x)=(\lambda^{r_1}x_1, \lambda^{r_2}x_2,..., \lambda^{r_N}x_N)$ for $r_i>0,\,\forall\, i=1,2,...,N$. The map $D_\lambda$ is called a {\it dilation} on $\mathbb{G}$. 
\end{definition}
For simplicity, we sometimes prefer to use the notation $\lambda x$ to denote the dilation $D_{\lambda}x$. Note that, if $\lambda x$ is a dilation then $\lambda^r x$ is also a dilation. The number $Q=r_1+r_2+...+r_N$ is called the homogeneous dimension of the homogeneous  Lie group $\mathbb{G}$ and the natural number $N$ represents the topological dimension of $\mathbb{G}.$ The Haar measure on $\mathbb{G}$ is denoted by $dx$ and it is nothing but the usual Lebesgue measure on $\mathbb{R}^N.$

\begin{definition}\label{d-carnot}
	A homogeneous Lie group $\mathbb{G} = (\mathbb{R}^N, \circ)$ is called a stratified Lie group (or a homogeneous Carnot group) if the following two conditions are fulfilled:
	\begin{enumerate}[label=(\roman*)]
		\item For some natural numbers $N_1 + N_2+ ... + N_k = N$ the decomposition $\mathbb{R}^N = \mathbb{R}^{N_1} \times \mathbb{R}^{N_2} \times ... \times \mathbb{R}^{N_k}$ holds, and for each $\lambda > 0$ there exists a dilation of the form $D_{\lambda}(x)=  (\lambda x^{(1)}, \lambda^2 x^{(2)},..., \lambda^{k}x^{(k)})$ which is an automorphism of the group $\mathbb{G}$. Here $x^{(i)}\in \mathbb{R}^{N_i}$ for $i = 1,2, ..., k$.
		
		\item With $N_1$  as in the above decomposition of $\mathbb{R}^N$, let $X_1, ..., X_{N_1}$ be the left invariant vector fields on $\mathbb{G}$ such that $X_k(0) = \frac{\partial}{\partial x_k}|_0$ for $k = 1, ..., N_1$. Then the H\"{o}rmander condition $rank(Lie\{X_1, ..., X_{N_1} \}) = N$
		holds for every $x \in \mathbb{R}^{N}$. In other words, the Lie algebra corresponding to the Lie group $\mathbb{G}$ is spanned by the iterated commutators of $X_1, ..., X_{N_1}$.
	\end{enumerate}
\end{definition}

Here $k$ is called the step of the stratified Lie group.  Note that, in this case, the homogeneous dimension becomes $Q=\sum_{i=1}^{i=k}iN_i$. Furthermore, the left-invariant vector fields $X_j$ satisfy the divergence theorem and they can be written explicitly as 
\begin{equation}\label{d-vec fld}
	X_i=\frac{\partial}{\partial x_i^{(1)}} + \sum_{j=2}^{k}\sum_{l=1}^{N_1} a^{(j)}_{i,l}(x^1, x^2, ..., x^{j-1})\frac{\partial}{\partial x_l^{(j)}}.
\end{equation}

For simplicity, we set $n=N_1$ in  the above Definition \ref{d-carnot}. 

An absolutely continuous curve $\gamma:[0,1]\rightarrow\mathbb{R}$ is said to be admissible, if there exist functions $c_i:[0,1]:\rightarrow\mathbb{R}$, for $i=1,2,...,n,$ such that
$${\dot{\gamma}(t)}=\sum_{i=1}^{i=n}c_i(t)X_i(\gamma(t))~\text{and}~ \sum_{i=1}^{i=n}c_i(t)^2\leq1.$$
Observe that the functions $c_i$ may not be unique as the vector fields $X_i$ may not be linearly independent. For any $x,y\in\mathbb{G}$ the Carnot-Carath\'{e}odory distance is defined as
$$\rho_{cc}(x,y)=\inf\{l>0: ~\text{there exists an admissible}~ \gamma:[0,l]\rightarrow\mathbb{G} ~\text{with}~ \gamma(0)=x ~\text{\&}~ \gamma(l)=y \}.$$
 The H\"{o}rmander condition for the vector fields $X_1,X_2,...X_{N_1}$ ensures that $\rho_{cc}$ is a metric. The space $(\mathbb{G}, \rho_{cc})$ is is known as a Carnot-Carath\'{e}odory space.


Let us now define the quasi-norm on the stratified Lie group $\mathbb{G}$.
\begin{definition}\label{d-quasi-norm}
	A continuous function $|\cdot|: \mathbb{G} \rightarrow \mathbb{R}^{+}$ is said to be a homogeneous quasi-norm on a stratified Lie group $\mathbb{G}$ if it satisfies the following conditions:
	\begin{enumerate}[label=(\roman*)]
		
		\item (definiteness): $|x| = 0$ if and only if $x = 0$.
		\item (symmetric): $|x^{-1}| = |x|$ for all $x \in\mathbb{G}$, and 
		\item ($1$-homogeneous): $|\lambda x| = \lambda |x|$ for all $x \in\mathbb{G}$ and $\lambda>0$.
		
	\end{enumerate}
\end{definition}

An example of a quasi-norm on $\mathbb{G}$ is the norm defined as $d(x):=\rho_{cc}(x, 0),\,\, x\in \mathbb{G}$, where $\rho$ is the Carnot-Carath\'{e}odory distance related to H\"ormander vector fields on $\mathbb{G}.$ It is known that all homogeneous quasi-norms are equivalent on $\mathbb{G}.$ In this paper we will work with a left-invariant homogeneous distance $d(x, y):=|y^{-1} \circ x|$ for all $x, y \in \mathbb{G}$ induced by the homogeneous quasi-norm of $\mathbb{G}.$

The sublaplacian (or Horizontal Laplacian)  on $\mathbb{G}$ is defined as
\begin{equation}\label{d-sub-lap}
	\mathcal{L}:=X_{1}^{2}+\cdots+X_{N_1}^{2}.
\end{equation}
The horizontal gradient on $\mathbb{G}$ is defined as
\begin{equation}\label{d-h-grad}
	\nabla_{\mathbb{G}}:=\left(X_{1}, X_{2}, \cdots, X_{N_1}\right).
\end{equation}
The horizontal divergence on $\mathbb{G}$ is defined by
\begin{equation}\label{d-h-div}
	\operatorname{div}_{\mathbb{G}} v:=\nabla_{\mathbb{G}} \cdot v.
\end{equation}
For $p \in(1,+\infty)$, we define the $p$-sublaplacian on the stratified Lie group $\mathbb{G}$ as
\begin{equation}\label{d-p-sub}
	\Delta_{\mathbb{G},p} u:=\operatorname{div}_{\mathbb{G}}\left(\left|\nabla_{\mathbb{G}} u\right|^{p-2} \nabla_{\mathbb{G}} u\right).
\end{equation}

Let $\Omega$ be a Haar measurable subset of $\mathbb{G}$. Then $\mu(D_{\lambda}(\Omega))=\lambda^{Q}\mu(\Omega)$ where $\mu(\Omega)$ is the Haar measure of $\Omega$. The quasi-ball of radius $r$ centered at $x\in\mathbb{G}$ with respect to the quasi-norm $|\cdot|$ is defined as
\begin{equation}\label{d-ball}
	B(x, r)=\left\{y \in \mathbb{G}: \left|y^{-1} \circ x\right|<r\right\}.
\end{equation}
Observe that $B(x, r)$ can be obtained by the left translation by $x$ of the ball $B(0, r)$. Furthermore, $B(0, r)$ is the image under the dilation $D_{r}$ of $B(0,1)$. Thus, we have $\mu(B(x,r))=r^Q\mu(B(0,1))$ for all $x\in \mathbb{G}$.



We are now in a position to define the notion of fractional Sobolev-Folland-Stein type spaces related to our study.

Let $\Omega \subset {\mathbb{G}}$ be an open subset. Then for $0<s<1< p<\infty$, the fractional Sobolev space $W^{s,p}(\Omega)$ on stratified groups is defined as 

\begin{equation}
	W^{s,p}(\Omega)=\{u\in L^{p}(\Omega): [u]_{s, p,\Omega}<\infty\},
\end{equation}
endowed with the norm 
\begin{equation}
	\|u\|_{W^{s,p}(\Omega)}^p=\|u\|_{L^p(\Omega)}^p+[u]_{s,p,\Omega}^p,
\end{equation}
where $[u]_{s, p,\Omega}$ denotes the Gagliardo semi-norm defined by
\begin{equation}
	[u]_{s, p,\Omega}:=\left(\int_{\Omega} \int_{\Omega} \frac{|u(x)-u(y)|^{p}}{\left|y^{-1} x\right|^{Q+ps}} dxdy\right)^{\frac{1}{p}}<\infty.
\end{equation}

Observe that for all $\phi \in C_c^{\infty}(\Omega)$, we have $[u]_{s, p,\Omega}<\infty$. We define the space  $W_0^{s,p}(\Omega)$ as the closure of $C_c^{\infty}(\Omega)$ with respect to the  norm $\|u\|_{W^{s,p}(\Omega)}$. We would like to point out that $W_0^{s,p}(\mathbb{G})=W^{s,p}(\mathbb{G})$ (see \cite{GKR23}).

Now for an open bounded subset $\Omega \subset {\mathbb{G}}$, define the space $X_0^{s,p}(\Omega)$  as the closure of $C_c^{\infty}(\Omega)$ with respect to the  norm $\|u\|_{L^p(\Omega)}+[u]_{s, p,\mathbb{G}}$. Note that the spaces $X_0^{s,p}(\Omega)$ and $W_0^{s,p}(\Omega)$ are different even in the Euclidean case unless $\Omega$ is an extension domain (see \cite{NPV12}). Note that when $\Omega=\mathbb{G}$, then we have $X_0^{s,p}(\mathbb{G})=W_0^{s,p}(\mathbb{G})=W^{s,p}(\mathbb{G})$. The space $X_0^{s,p}(\Omega)$ is a  reflexive Banach space for $1< p<\infty$.

The space $X_0^{s,p}(\Omega)$ can be defined as a closure of $C_c^{\infty}(\Omega)$ with respect to the homogeneous norm $[u]_{s,p, \mathbb{G}}$. That is $$\|u\|_{X_0^{s,p}(\Omega)}\cong[u]_{s, p,\mathbb{G}}~\text{for all}~u\in X_0^{s,p}(\Omega).$$
	
	Moreover, the construction of the space $X_0^{s,p}(\Omega)$ suggests that if $\Omega$ is bounded then we can represent $X_0^{s,p}(\Omega)$ as follows:  $$X_0^{s,p}(\Omega)=\{u\in W^{s, p}({\mathbb{G}}): u=0~\text{in}~{\mathbb{G}}\setminus\Omega \}.$$ 
For $u\in X_0^{s,p}(\Omega)$,
\begin{align*}
    [u]_{s, p,\mathbb{G}}^p&=\iint_{\mathbb{G} \times \mathbb{G}} \frac{|u(x)-u(y)|^{p}}{\left|y^{-1} x\right|^{Q+p s}} dxdy=\iint_{\mathbb{G} \times \mathbb{G}\setminus (\Omega^c\times\Omega^c)} \frac{|u(x)-u(y)|^{p}}{\left|y^{-1} x\right|^{Q+p s}} dxdy.
\end{align*}


We conclude this section with the following two definitions. For $s\in(0,1)$ and $p\in(1,\infty)$, we define the fractional $p$-sublaplacian  as 
\begin{equation}
	\left(-\Delta_{p,{\mathbb{G}}}\right)^s u(x):=C_{Q,s, p}\,\,  P.V. \int_{{\mathbb{G}} } \frac{|u(x)-u(y)|^{p-2}(u(x)-u(y))}{\left|y^{-1} x\right|^{Q+p s}} dy, \quad x \in {\mathbb{G}}.
\end{equation}

For any $\varphi\in X_0^{s,p}(\Omega)$, we have

\begin{equation}
	\langle\left(-\Delta_{p,{\mathbb{G}}}\right)^s u,\varphi\rangle= \iint_{\mathbb{G} \times \mathbb{G}} \frac{|u(x)-u(y)|^{p-2}(u(x)-u(y))(\varphi(x)-\varphi(y))}{\left|y^{-1} x\right|^{Q+p s}} dxdy.
\end{equation}
Let us recall the following theorem on the subelliptic fractional embedding results proved in \cite{GKR23}.
\begin{theorem}\cite[Theorem 1.1]{GKR23} \label{l-3} Let $\mathbb{G}$ be a stratified Lie group of homogeneous dimension $Q$, and let $\Omega\subset\mathbb{G}$ be an open set.
Let $0<s<1\leq  p<\infty$ and $Q>sp.$  Then the fractional Sobolev space $X_0^{s, p}(\Omega)$ is continuously embedded in $L^r(\Omega)$ for $p\leq r\leq p_s^*:=\frac{Qp}{Q-sp}$, that is, there exists a positive constant $C=C(Q,s,p, \Omega)$ such that for all $u\in X_0^{s, p}(\Omega)$, we have
\begin{equation}\label{localemb}
    \|u\|_{L^r(\Omega)}\leq C \|u\|_{X_0^{s,p}(\Omega)}
\end{equation}
In particular, we have 
\begin{equation} \label{emG}
    \|u\|_{L^r(\mathbb{G})}\leq C_{Q, s, p} \|u\|_{W^{s,p}(\mathbb{G})}.
\end{equation}

Moreover, if $\Omega$ is bounded, then the following embedding
\begin{align} \label{compactemG}
    X_0^{s,p}(\Omega) \hookrightarrow L^r(\Omega)
  \end{align}
is continuous for all $r\in[1,p_s^*]$ and is compact for all $r\in[1,p_s^*)$.
\end{theorem}

\section{A subelliptic version of Lions vanishing lemma for $W_0^{s, p}(\G)$} \label{secc3}
In this section, we prove  the subelliptic vanishing lemma established for classical Sobolev space by Lions \cite{Li85I,Li85II} in his seminal papers. 
Recently, this lemma was obtained by Wang \cite{wang23} in the setting of the Heisenberg group. We follow the proof presented in a recent paper \cite{zhang} dealing with the Euclidean space scenario. 
\begin{lemma} \label{vanishinglemma}
Let $\G$ be a stratified Lie group with homogeneous dimension $Q.$
    Let $s \in (0, 1), p \in (1, \infty),$ and let $q$ be such that  $p\leq q<p_s^*=\frac{Qp}{Q-ps}.$ Let $(u_k)$ be a bounded sequence in $W^{s,p}_0(\G)$ with the property 
    \begin{equation} \label{hy4.1}
        \liminf_{k \rightarrow \infty} \sup_{x \in \G} \int_{B(x, 1)} |u_k|^q dy =0.
    \end{equation}
     Then there exists a subsequence, again denoted by $(u_k),$ such that $u_k \rightarrow 0$ in $L^r(\G)$ for $r \in (q, p_s^*).$
\end{lemma}
\begin{proof}  Let $(u_k)$ be a sequence in  $ W^{s, p}_0(\G)$ and $q \in [p, p_s^*).$ We will first prove that 
\begin{align} \label{lab4.2}
    \int_{\G} |u_k|^r dy  \leq C \Big( \sup_{x \in \G} \int_{B(x, 1)} |u_k|^q dy \Big)^{\frac{p_s^*-r}{p_s^*-q}} \|u_k\|^p_{W^{s, p}_0(\G)}
    \end{align} for any $r \in (q, p_s^*)$ and for all $u_k \in W^{s, p}_0(\G).$

For any $r \in (q, p_s^*)$ we write $\frac{1}{r}=\frac{\beta}{q}+\frac{1-\beta}{p_s^*}$ for some $\beta \in (0, 1).$ Then, using the H\"older inequality we have 
$$\|u_k\|_{L^r(B(x, 1))}^r \leq \|u_k\|_{L^q(B(x, 1))}^{q \beta  } \|u_k\|_{L^{p_s^*}(B(x, 1))}^{(1-\beta)p_s^* }. $$
Now, by choosing $\beta= \frac{p_s^*-r}{p_s^*-q} \in (0, 1)$  and using the local Sobolev embedding \eqref{localemb} we get
\begin{align*}
    \int_{B(x, 1)} |u_k|^r dy &\leq \Bigg( \int_{B(x, 1)} |u_k|^q dy \Bigg)^{\frac{p_s^*-r}{p_s^*-q}} \Bigg( \int_{B(x, 1)} |u_k|^{p_s^*} dy \Bigg)^{\frac{r-q}{p_s^*-q}}  \\&
    \leq C \Bigg( \int_{B(x, 1)} |u_k|^q dy \Bigg)^{\frac{p_s^*-r}{p_s^*-q}} \|u_k\|_{X^{s, p}_0(B(x, 1))}^{p_s^*\frac{r-q}{p_s^*-q}}.
\end{align*}
Next, for this choice of $\beta$ we have $r = \frac{p(p_s^*-q)+p_s^*q}{p_s^*}$ so that  $\frac{p_s^*}{p}\frac{r-q}{p_s^*-q}=p.$ Thus we get
\begin{align}
    \int_{B(x, 1)} |u_k|^r dy \leq C  \Bigg( \int_{B(x, 1)} |u_k|^q dy \Bigg)^{\frac{p_s^*-r}{p_s^*-q}} \|u_k\|_{X^{s, p}_0(B(x, 1))}^{p}.
\end{align}
   Using the covering lemma \cite[Lemma 5.7.5]{FR16}, we can choose a covering $\{B(x_i, 1)\}_{i=1}^\infty$ of $\G$ such that no point of $\G$ belongs to more than $(4C_0^2)^Q$ such balls, where $C_0$ is the constant appearing in the triangle inequality of the quasi-norm $|\cdot|$ on $\G$.
Therefore, we have
\begin{align}
    \int_{\G} |u_k|^r dy &\leq \int_{\cup_{i} B(x_i, 1)} |u_k|^r dy  \nonumber  \leq 
    \sum_{i} \int_{B(x_i, 1)} |u_k|^r dy  \nonumber \\& \leq C \sum_{i} \Bigg( \int_{B(x_i, 1)} |u_k|^q dy \Bigg)^{\frac{p_s^*-r}{p_s^*-q}} \|u_k\|_{X^{s, p}_0(B(x_i, 1))}^{p} \nonumber \\&\leq C (4C_0)^Q \Bigg( \sup_{x \in \G} \int_{B(x, 1)} |u_k|^q dy \Bigg)^{\frac{p_s^*-r}{p_s^*-q}} \sum_{i} \|u_k\|_{X^{s, p}_0(B(x_i, 1))}^{p} \nonumber \\&\leq C (4C_0)^Q\Big( \sup_{x \in \G} \int_{B(x, 1)} |u_k|^q dy\Big)^{\frac{p_s^*-r}{p_s^*-q}} \|u_k\|^p_{W^{s, p}_0(\G)},
\end{align} 
proving \eqref{lab4.2}.
Thus, applying the hypothesis \eqref{hy4.1} in \eqref{lab4.2} we conclude that 

$$\liminf_{k \rightarrow \infty} \int_{\G} |u_k|^rdy =0,$$
for $r \in (q, p_s^*).$ This implies that $u_k \rightarrow 0$ in $L^r(\G)$  for $r \in (q, p_s^*).$ \end{proof}


\section{Existence of least energy solutions for \eqref{r1.71}}\label{sec3}

This section is devoted to proving Theorem \ref{rt1.2i}. Throughout this section,  we always assume that $p<q<p_s^*$. For the convenience of readers, let us recall  the subelliptic  fractional stationary Schr\"odinger equation \eqref{r1.71} which is the main object of study in Theorem \ref{rt1.2i}:

\begin{equation}\label{r1.7}
	\left(-\Delta_{p, \mathbb{G}}\right)^s u+|u|^{p-2} u = |u|^{q-2} u, \quad u \in W_{0}^{s,p}(\mathbb{G}),
\end{equation}
where $p<q<p_s^*:=\frac{Qp}{Q-ps}$. 
Next, we will discuss several definitions which would be crucial for the proof of Theorem \ref{rt1.2i} below. We begin by defining the notion of a weak solution for the nonlocal equation \eqref{r1.7}.

\begin{definition}
    We say that $u \in W_{0}^{s,p}(\mathbb{G})$ is a weak solution of \eqref{r1.7} if and only if for all $\phi \in W_{0}^{s,p}( \mathbb{G})$, we have

\begin{equation}\label{r1.8}
	\int_{\mathbb{G}}\int_{\mathbb{G}}\frac{|u(x)-u(y)|^{p-2}(u(x)-u(y))(\phi(x)-\phi(y)}{|y^{-1}x|^{Q+ps}}dxdy +  \int_{ \mathbb{G}}\left(|u|^{p-2}u\phi-|u|^{q-2} u \phi\right) dx=0.
\end{equation}
\end{definition}

The  energy functional $I: W^{s,p}_0(\G) \rightarrow \mathbb{R}$ associated to the problem \eqref{r1.7}  is defined by 

\begin{equation*}
	I(u):=\frac{1}{p} \iint_{\G \times \G} \frac{|u(x)-u(y)|^p}{|y^{-1} x|^{Q+ps}} dx dy +\frac{1}{p} \int_{\G} |u(x)|^{p} dx -\frac{1}{q} \int_{ \G}|u(x)|^{q} dx
\end{equation*}

such that
\begin{align*}
   \langle I'(u),\phi \rangle&= \int_{\mathbb{G}}\int_{\mathbb{G}}\frac{|u(x)-u(y)|^{p-2}(u(x)-u(y))(\phi(x)-\phi(y)}{|y^{-1}x|^{Q+ps}}dxdy \\&\quad\quad\quad+  \int_{ \mathbb{G}}\left(|u|^{p-2}u\phi-|u|^{q-2} u \phi\right) dx,
\end{align*}
for every $\phi \in W^{s,p}_0(\G).$ In particular, by setting $\mathcal{L}(u):=\langle I'(u), u \rangle,$ we get  
$$\mathcal{L}(u):=\int_{\mathbb{G}}\int_{\mathbb{G}}\frac{|u(x)-u(y)|^{p}}{|y^{-1}x|^{Q+ps}}dxdy +  \int_{ \mathbb{G}}\left(|u|^{p}-|u|^{q} \right) dx. $$

Denote the Nehari manifold by 
\begin{equation}\label{neh}
    \mathcal{N}:=\left\{u \in W_{0}^{s,p}(\G) \setminus\{0\}: \mathcal{L}(u)=\langle I'(u),u\rangle=0\right\}
\end{equation}
and set
\begin{equation}\label{r1.9}
	d:=\inf \{I(u): u \in \mathcal{N}\} .
\end{equation}
Since it is clear that the weak solutions of \eqref{r1.7} are the critical points of the energy functional $I$, we define the following set:
\begin{definition}\label{rd1.1}
	Let $\Gamma$ be the set of weak solutions of \eqref{r1.7}, namely, we define	
	\begin{equation*}
		\Gamma:=\left\{\phi \in W_{0}^{s,p}( \G): I'(\phi)=0 \text { and } \phi \neq 0\right\}.
	\end{equation*}
	
	Let $\mathcal{G}$ be the set of least energy solutions of \eqref{r1.7}, that is,
	\begin{equation*}
		\mathcal{G}:=\{u \in \Gamma: I(u) \leq I(v) \text { for any } v \in \Gamma\}.
	\end{equation*}
\end{definition}

 The following supporting lemmas will be useful in what follows.

\begin{lemma}\label{rl2.1}
	For any $u \in W_{0}^{s,p}( \G) \setminus\{0\}$, there is a unique $\theta_{u}>0$ such that $\theta_{u} u \in$ $\mathcal{N}$. Moreover, if $\mathcal{L}(u)<0$, then $0<\theta_{u}<1$.
\end{lemma}
\begin{proof}
	For any arbitrary $u \in W_{0}^{s,p}( \G) \setminus\{0\}$ and 
	\begin{equation*}
		\theta_{u}=\|u\|_{W_{0}^{s,p}( \G)}^{\frac{p}{q-p}}\|u\|_{L^{q}(\G)}^{-\frac{q}{q-p}},
	\end{equation*}
 a simple calculation yields that
	 $\theta_{u} u \in \mathcal{N}$. It is easy to see that this $\theta_{u}$ is unique. From above and the expression of $\mathcal{L}(u)=\|u\|_{W_{0}^{s,p}( \G)}^{p}-\|u\|_{L^{q}(\G)}^{q}$, we know that if $\mathcal{L}(u)<0$, that is, $\|u\|_{W_{0}^{s,p}( \G)}^{p}<\|u\|_{L^{q}(\G)}^{q}$, then $0<\theta_{u}<1$.
\end{proof} 

\begin{lemma}\label{rl2.2}
	There is $\rho>0$ such that for all $u \in \mathcal{N}$ we have $\|u\|_{W_{0}^{s,p}( \G)} \geq \rho$.
\end{lemma}
\begin{proof}
Since $u \in \mathcal{N},$ we have $\|u\|_{W_{0}^{s,p}( \G)}^{p}=\|u\|_{L^{q}(\G)}^{q}.$ Therefore,  from the subelliptic fractional Gagliardo-Nirenberg inequality \eqref{r1.6} and subelliptic fractional Sobolev inequality \eqref{emG} we obtain
\begin{align*}
    \|u\|_{W_{0}^{s,p}( \G)}^{p}&=\|u\|_{L^{q}(\G)}^{q} \leq C\left(\int_{\mathbb{G}}\int_{\mathbb{G}}\frac{|u(x)-u(y)|^{p}}{\left|y^{-1} x\right|^{Q+ps}} dxdy\right)^{\frac{Q(q-p)}{sp^2}}\left(\int_{ \mathbb{G}}|u|^{p} dx\right)^{\frac{spq-Q(q-p)}{sp^2}}
 \\&\leq \|u\|_{W^{s,p}_0(\G)}^{\frac{Q(q-p)}{sp}} \|u\|_{L^p(\G)}^{\frac{spq-Q(q-p)}{sp}} \leq C_{Q, s, p}\|u\|_{W^{s,p}_0(\G)}^{\frac{Q(q-p)}{sp}} \|u\|_{W^{s,p}_0(\G)}^{\frac{spq-Q(q-p)}{sp}}= C_{Q, s, p} \|u\|_{W^{s,p}_0(\G)}^{q}, 
\end{align*}
	which implies that $\|u\|_{W^{s,p}_0(\G)}^{p-q} \geq C_{Q, s, p}^{-1}$. Setting $\rho=C_{Q, s, p}^{-\frac{1}{p-q}}>0$, we get that $\|u\|_{W^{s,p}_0(\G)} \geq \rho$ for all $u \in \mathcal{N}$.
\end{proof}

\begin{lemma}\label{rl2.4}
	Suppose that $v \in \mathcal{N}$ and $I(v)=d:=\inf \{I(u): u \in \mathcal{N}\}.$ Then $v$ is a least energy solution of \eqref{r1.7}.
\end{lemma}
\begin{proof} Let $I(v)=d$ for $v \in \mathcal{N}.$ This implies that 
	 $v$ is a minimiser for  $d$. Therefore,  from the Lagrange multiplier rule, we infer that there is $\theta \in \mathbb{R}$ such that for any $\psi \in W_{0}^{s,p}( \G)$,
	
	\begin{equation} \label{theq}
		\Big\langle I'(v), \psi\Big\rangle_{\G}=\theta\left\langle \mathcal{L}^{'}(v), \psi\right\rangle_{\G}.
	\end{equation}
	Here $\langle \cdot, \cdot \rangle_{\G}$ is a dual product between $W^{s,p}_0(\G)$ and its dual space. 
	Note that as $q>p$ we conclude, noting that $v \in \mathcal{N}$ and so $\mathcal{L}(v)=\|u\|_{W_{0}^{s,p}( \G)}^{p}-\|u\|_{L^{q}(\G)}^{q}=0,$ that
	
	\begin{equation*}
		\left\langle \mathcal{L}^{'}(v), v\right\rangle_{\G}=p\|v\|_{W^{s,p}_0(\G)}^{p}-q \int_{\G}|v|^{q} dx=(p-q) \int_{\G}|v|^{q} dx<0.
	\end{equation*}
	
 In addition, $\left\langle I'(v), v\right\rangle_{\G}=\mathcal{L}(v)=0$ by the choice of $v$. Consequently, \eqref{theq} forces   that $\theta=0$. Hence $I'(v)=0$. From Definition \ref{rd1.1}, one sees easily that $v$ is a least energy solution of \eqref{r1.7}.
	\end{proof} 
\noindent Having these supporting lemmas in hand, we are ready to prove Theorem \ref{rt1.2i}.
\begin{proof}[Proof of Theorem \ref{rt1.2i}]
     Let us choose $\left(v_{n}\right) \subset \mathcal{N}$ as a minimising sequence. Then, from the Ekeland variational principle,  we deduce that there is a sequence $\left(u_{n}\right) \subset \mathcal{N}$ such that
\begin{equation*}
	I(u_{n}) \rightarrow d \quad \text { and } \mathcal{L}(u_{n}) \rightarrow 0 .
\end{equation*}

The subellitic fractional Sobolev inequality \eqref{emG} and Lemma \ref{rl2.2} imply that there exist two positive constants $C_{1}$ and $C_{2}$ such that

\begin{equation*}
	C_{1} \leq\left\|u_{n}\right\|_{W^{s, p}_0(\G)} \leq C_{2} .
\end{equation*}

Combining this with the fact that $\left\|u_{n}\right\|_{W^{s, p}_0(\G)}^{p}=\int_{\G}\left|u_{n}\right|^{q} dx$ as $(u_n)_n \subset \mathcal{N}$, we have that there exists a positive constant $C_{3}$ such that
\begin{equation}\label{r2.1}
	\limsup _{n \rightarrow \infty} \int_{\G}\left|u_{n}(x)\right|^{q} dx \geq C_{3}>0 .
\end{equation}

Using the subelliptic Lions vanishing lemma (see Lemma \ref{vanishinglemma}), we have that $u_{n} \rightarrow 0$ in $L^{q}\left(\G\right)$ for any $p<q<p_s^*$ provided that 
\begin{equation*}
	\liminf _{n \rightarrow \infty} \sup _{y \in  \G} \int_{B(y, 1)}\left|u_{n}(x)\right|^{q} dx=0.
\end{equation*}
 Therefore,  \eqref{r2.1} implies that there exists $C_{4}>0$  such that

\begin{equation}\label{r2.2}
	\liminf _{n \rightarrow \infty} \sup _{y \in  \G} \int_{B(y, 1)}\left|u_{n}(x)\right|^{q} dx \geq C_{4}>0 .
\end{equation}

Therefore, we may assume that there is a $\tilde{y} \in  \G$ such that
\begin{equation}\label{r2.3}
	\liminf _{n \rightarrow \infty} \int_{B\left(\tilde{y}, 1\right)}\left|u_{n}(x)\right|^{p} dx \geq \frac{C_{4}}{2}>0 .
\end{equation}

Using the bi-invariance property of the Haar measure $dx$ on $\G,$ a direct computation yields that

\begin{equation*}
	I(u_{n}(y \circ \tilde{y}))=I(u_{n}(y))  
\end{equation*}
and 
\begin{equation*}
     \mathcal{L}(u_{n}(y \circ \tilde{y}))=\mathcal{L}(u_{n}(y))
\end{equation*}
for all $\tilde{y} \in \G.$
Let us define a new sequence $w_{n}(y):=u_{n}(y \circ \tilde{y})$. It is clear  that $I(w_{n})=I(u_{n})$ and $\mathcal{L}(w_{n})=\mathcal{L}(u_{n})$. Moreover, it provides that $(w_{n})$ is bounded in $W_{0}^{s,p}( \G)$ satisfying

\begin{equation}\label{r2.4}
	\liminf _{n \rightarrow \infty} \int_{B(0, 1)}\left|w_{n}(x)\right|^{q} dx \geq \frac{C_{4}}{2}>0 .
\end{equation}

Passing to a subsequence we may assume $w_{n} \rightarrow \phi$ weakly in $W_{0}^{s,p}( \G)$. Theorem \ref{l-3} implies that $w_{n} \rightarrow \phi$ strongly in $L_{\text{loc}}^{q}( \G)$. Combining this with \eqref{r2.4}, we know that $\phi \neq 0$. 

Next we will show that $(w_{n})$  converges to $\phi$ strongly in $W_{0}^{s,p}( \G)$. For this purpose, we first prove that $\mathcal{L}(\phi)=0.$ Contrary to this, let us consider the first case, if possible, $\mathcal{L}(\phi)<0$. Then, Lemma \ref{rl2.1} immediately implies  that there exits a $0<\theta_{\phi}<1$ such that $\theta_{\phi} \phi \in \mathcal{N}$. Therefore using the Fatou's  lemma and the fact that $\mathcal{L}(w_{n})=0$, we obtain
\begin{align}\label{r2.5}
		d+o(1) &=I(w_{n})=\left(1-\frac{1}{q}\right) \int_{\G}\left|w_{n}\right|^{q} dx -\left(1-\frac{1}{p}\right) \int_{\G}\left|w_{n}\right|^{p} dx \nonumber\\& \geq \left(1-\frac{1}{q}\right) \int_{\G}\left|w_{n}\right|^{q} dx \nonumber \\&\geq \left(1-\frac{1}{q}\right) \int_{\G}\left|\phi\right|^{q} dx +o(1) =\left(1-\frac{1}{q}\right) \theta_{\phi}^{-q} \int_{\G}\left|\theta_{\phi} \phi\right|^{q} dx+o(1)\nonumber\\
		&\geq\left(1-\frac{1}{q}\right) \theta_{\phi}^{-q} \int_{\G}\left|\theta_{\phi} \phi\right|^{q} dx -\left(1-\frac{1}{p}\right)\theta_{\phi}^{-q} \int_{\G}\left|\theta_{\phi}\phi\right|^{p} dx +o(1) \nonumber\\&=\theta_{\phi}^{-q} I(\theta_{\phi} \phi)+o(1) .
	\end{align}
We now deduce from $0<\theta_{\phi}<1$ that $I(\theta_{\phi} \phi)<d$, which is a contradiction  as $\theta_{\phi} \phi \in \mathcal{N}$.

Now, we consider the remaining case,  $\mathcal{L}(\phi)>0$. Then, from  Brezis-Lieb lemma \cite[Theorem 1.3.5]{BL83} with $\psi_{n}=w_{n}-\phi$ we have that
\begin{equation*}
	0=\mathcal{L}(w_{n})=\mathcal{L}(\phi)+\mathcal{L}(\psi_{n})+o(1).
\end{equation*}

As a consequence, along with the fact that $\mathcal{L}(\phi)>0$, one deduces that

\begin{equation}\label{r2.6}
	\limsup _{n \rightarrow \infty} \mathcal{L}\left(\psi_{n}\right)<0 .
\end{equation}

By Lemma \ref{rl2.1}, there is positive constant $\theta_{n}:=\theta_{\psi_{n}}$ such that $\theta_{n} \psi_{n} \in \mathcal{N}$. Moreover, we claim that $\lim \sup _{n \rightarrow \infty} \theta_{n} \in(0,1)$. Indeed, suppose that $\lim \sup _{n \rightarrow \infty} \theta_{n}=1$. Then, we can extract a subsequence $\left(\theta_{n_{j}}\right)$ such that $\lim _{j \rightarrow \infty} \theta_{n_{j}}=1$. Thus from the property $\theta_{n_{j}} \psi_{n_{j}} \in \mathcal{N}$, we obtain that $\mathcal{L}\left(\psi_{n_{j}}\right)=\mathcal{L}\left(\theta_{n_{j}} \psi_{n_{j}}\right)+o(1)=o(1)$. This is not possible in view of  \eqref{r2.6}. Therefore, we obtain the required conclusion that $\underset{n \rightarrow \infty}{\limsup}\, \theta_{n} \in(0,1)$. Since $\mathcal{L}(w_n)=0,$ similar to \eqref{r2.5}, using the Fatou's lemma we get
\begin{equation*}
	\begin{aligned}
		d+o(1) &=I\left(w_{n}\right)= \left(1-\frac{1}{q}\right) \int_{\G}\left|w_{n}\right|^{q} dx   -\left(1-\frac{1}{p}\right) \int_{\G}\left|w_{n}\right|^{p} dx \nonumber \\&\geq \nonumber  \left(1-\frac{1}{q}\right) \int_{\G}\left|w_{n}\right|^{q} dx  \geq\left(1-\frac{1}{q}\right) \int_{\G}\left|\psi_{n}\right|^{q} dx+o(1) \\
		&\geq \left(1-\frac{1}{q}\right) \theta_{n}^{-q} \int_{\G}\left|\theta_{n} \psi_{n}\right|^{q} dx-\left(1-\frac{1}{p}\right) \theta_{n}^{-q} \int_{\G}\left|\theta_{n} \psi_{n}\right|^{p} dx+o(1) \nonumber
  \\&=\theta_{n}^{-q} I\left(\theta_{n} \psi_{n}\right)+o(1).
	\end{aligned}
\end{equation*}

This along with the fact that $\underset{n \rightarrow \infty}{\lim \sup}\, \theta_{n} \in(0,1)$ implies that  $d>I\left(\theta_{n} \psi_{n}\right)$, which is a contradiction  as $\theta_{n} \psi_{n} \in \mathcal{N}$.

By joining both cases together, we conclude that
 $\mathcal{L}(\phi)=0$. 
 
 We are now in a position to prove our claim that $\psi_{n}=w_n-\phi \rightarrow 0$ in $W_{0}^{s,p}( \G)$. Indeed, if possible, suppose that it is not true,  that means, $\left\|\psi_{n}\right\|_{W^{s,p}(\G)} \nrightarrow 0$ as $n \rightarrow \infty$. Then as an application of the continuous fractional Sobolev embedding \eqref{emG}, the following two situations are possible:

(i) When $\int_{\G}\left|\psi_{n}\right|^{q} dx \nrightarrow  0$ as $n \rightarrow \infty,$ we have from

\begin{equation*}
	0=\mathcal{L}\left(w_{n}\right)=\mathcal{L}(\phi)+\mathcal{L}\left(\psi_{n}\right)+o(1)=\mathcal{L}\left(\psi_{n}\right)+o(1)
\end{equation*}
and therefore, $\psi_n \in \mathcal{N}.$
The Brezis-Lieb lemma \cite[Theorem 1.3.5]{BL83} implies that

\begin{equation*}
	d+o(1)=I\left(w_{n}\right)=I(\phi)+I\left(\psi_{n}\right)+o(1) \geq d+d+o(1),
\end{equation*}

which is not true as $d>0$.

(ii) When $\int_{\G}\left|\psi_{n}\right|^{q} dx \rightarrow 0$ as $n \rightarrow \infty,$ we have that

\begin{equation*}
	d+o(1)=I\left(w_{n}\right)=I(\phi)+\frac{1}{p}\left\|\psi_{n}\right\|_{W^{s,p}_0(\G)}^{p}+o(1) \geq d+\frac{1}{p}\left\|\psi_{n}\right\|_{W^{s,p}_0(\G)}^{p}+o(1)>d
\end{equation*}
which is impossible. 
Thus, our assumption that $\left\|\psi_{n}\right\|_{W^{s,p}(\G)} \nrightarrow 0$ is not correct. Therefore,  we deduce that $(w_{n})$  converges strongly to $\phi$ in $W_{0}^{s,p}( \G)$ and $\phi$ is a minimiser for  $d$. Finally, Lemma \ref{rl2.4} implies that $\phi$ is a least energy solution of \eqref{r1.7}. This completes the proof. \end{proof}

\section{Best constants in the subelliptic fractional   Gagliardo-Nirenberg inequalities }\label{sec4}

Throughout this section, we continue to assume the same range of $q,$ that is,  $p<q<p_s^*$ as in the previous section. The aim of this section is to give a sharp estimate of $C_{GN, \G}$ in \eqref{r1.6}. Firstly, we give some properties of the least energy solutions of \eqref{r1.7} which will be used in what follows.

\begin{lemma}\label{rl3.1}
	Let $\phi$ be a least energy solution of \eqref{r1.7}. Then
	
	\begin{equation}\label{se2}
		\iint_{\mathbb{G} \times \mathbb{G}}\frac{|\phi(x)-\phi(y)|^{p}}{|y^{-1}x|^{Q+ps}}dxdy=\frac{Q(q-p)}{pqs-Q(q-p)} \int_{\G}|\phi|^{p} dx
  \end{equation}
  
  and  \begin{equation} \label{se1}
      \int_{\G}|\phi|^{q} dx=\frac{pqs}{pqs-Q(q-p)} \int_{\G}|\phi|^{p} dx.
	\end{equation}
\end{lemma}
\begin{proof}
	Since $\phi$ is a least energy solution of \eqref{r1.7}, we obtain from \eqref{r1.8} that
	
	\begin{equation} \label{eq5.2}
		\iint_{\mathbb{G} \times \mathbb{G}}\frac{|\phi(x)-\phi(y)|^{p}}{|y^{-1}x|^{Q+ps}}dxdy+\int_{\G} |\phi(x)|^{p} dx=\int_{\G}|\phi(x)|^{q} dx.
	\end{equation}
	
	Define a function $\tilde{\phi}_{\lambda}$  for $\lambda>0$ by $\tilde{\phi}_{\lambda}(x):=\lambda^{\frac{Q}{p}} \phi\left(D_{\lambda}(x)\right)$. Then, we have  
 \begin{align*}
      I(\tilde{\phi}_{\lambda})&= \frac{\lambda^Q}{p} \iint_{\G \times \G} \frac{|\tilde{\phi}_{\lambda}(x)-\tilde{\phi}_{\lambda}(y)|^p}{|y^{-1} x|^{Q+ps}} dx dy +\frac{\lambda^Q}{p} \int_{\G} |\tilde{\phi}_{\lambda}(x)|^{p} dx -\frac{\lambda^{\frac{Qq}{p}}}{q} \int_{ \G}|\tilde{\phi}_{\lambda}(x)|^{q} dx \\&= \frac{\lambda^{sp}}{p} \iint_{\G \times \G} \frac{|\phi(x)-\phi(y)|^p}{|y^{-1} x|^{Q+ps}} dx dy +\frac{1}{p} \int_{\G} |\phi(x)|^{p} dx -\frac{\lambda^{\frac{Qq}{p}-Q}}{q} \int_{ \G}|\phi(x)|^{q} dx.
 \end{align*}
 Then, it follows that 
	\begin{equation} \label{eq5.3}
		0=\left.\frac{\partial}{\partial \lambda} I\left(\tilde{\phi}_{\lambda}\right)\right|_{\lambda=1}= s\iint_{\G \times \G} \frac{|\phi(x)-\phi(y)|^p}{|y^{-1} x|^{Q+ps}} dx dy  -\frac{Q(q-p)}{pq} \int_{ \G}|\phi(x)|^{q} dx.
	\end{equation}
	Using \eqref{eq5.3} in \eqref{eq5.2} we get that 
 \begin{equation*}
     \int_{\G} |\phi(x)|^p dx = \Bigg(1 -\frac{Q(q-p)}{spq} \Bigg) \int_{ \G}|\phi(x)|^{q} dx,
 \end{equation*}
 which implies that 
	\begin{equation}
	    \int_{\G}|\phi|^{q} dx=\frac{pqs}{pqs-Q(q-p)} \int_{\G}|\phi|^{p} dx
	\end{equation}
 establishing \eqref{se1}. Using this in \eqref{eq5.2}, we get 
 $$\int_{\mathbb{G}}\int_{\mathbb{G}}\frac{|\phi(x)-\phi(y)|^{p}}{|y^{-1}x|^{Q+ps}}dxdy=\Bigg(\frac{pqs}{pqs-Q(q-p)}-1 \Bigg) \int_{\G}|\phi|^{p} dx= \frac{Q(p-q)}{pqs-Q(q-p)} \int_{\G}|\phi|^{p} dx,$$
 proving \eqref{se2}.
\end{proof}

\begin{lemma}\label{rl3.2} Let 
\begin{equation*}
	T_{\rho,p,q,s}:=\inf\left\{\|u\|^p_{W^{s,p}_0(\G)}: u \in W_{0}^{s,p}\left(\G\right)~ \text{ and }~ \int_{\G}|u|^{q} dx=\rho\right\}.
\end{equation*}
If $\phi$ is a minimiser obtained in Theorem \ref{rt1.2i}, then $\phi$ is a minimiser of $T_{\rho_{0}, p,q, s}$ with $\rho_{0}=\int_{\G}|\phi|^{q} dx$.
\end{lemma}
\begin{proof}
	It is clear from the definition of $T_{\rho,p,q,s}$ that $\|\phi\|^{p}_{W^{s,p}_0(\G)} \geq T_{\rho_{0},p,q,s}$. Now, for any $u \in W_{0}^{s,p}\left(\G\right)$ satisfying $\int_{\G}|u|^{q} dx=$ $\int_{\G}|\phi|^{q} dx$, by Lemma \ref{rl2.1} there is a unique
	
	\begin{equation*}
		\lambda_{0}=\|u\|_{W^{s,p}_0(\G)}^{\frac{p}{q-p}}\left(\int_{\G}|u|^{q} dx\right)^{-\frac{1}{q-p}}
	\end{equation*}
	such that $\lambda_{0} u \in \mathcal{N},$ which implies that $\mathcal{L}\left(\lambda_{0} u\right)=0$. Since $\lambda_{0} u \neq 0$ and $\phi$ achieves the minimum $d$, a simple calculation asserts that 
	
	\begin{align}
			\left(\frac{1}{p}-\frac{1}{q}\right)\|\phi\|^{p}_{W^{s,p}_0(\G)} &=I(\phi) \leq I\left(\lambda_{0} u\right)=\left(\frac{1}{p}-\frac{1}{q}\right) \lambda_{0}^{p}\|u\|^{p}_{W^{s,p}_0(\G)} \nonumber\\
			&=\left(\frac{1}{p}-\frac{1}{q}\right)\|u\|_{W^{s,p}_0(\G)}^{\frac{p^2}{q-p}}\left(\int_{\G}|u|^{q} dx\right)^{-\frac{p}{q-p}}\|u\|_{W^{s,p}_0(\G)}^{p}\nonumber\\&=\left(\frac{1}{p}-\frac{1}{q}\right) \|u\|_{W^{s,p}_0(\G)}^{\frac{pq}{q-p}} \left(\int_{\G}|u|^{q} dx\right)^{-\frac{p}{q-p}}.
		\end{align}

	Since $\int_{\G}|u|^{q} dx=\int_{\G}|\phi|^{q} dx$ and $\int_{\G}|\phi|^{q} dx=\|\phi\|^{p}_{W^{s,p}_0(\G)},$ we get that
 \begin{align}
     \left(\frac{1}{p}-\frac{1}{q}\right)\|\phi\|^{p}_{W^{s,p}_0(\G)} &\leq \left(\frac{1}{p}-\frac{1}{q}\right)\|u\|_{W^{s,p}_0(\G)}^{\frac{pq}{q-p}} \|\phi\|_{W^{s,p}_0(\G)}^{-\frac{p^2}{q-p}}
 \end{align} implying that 
 $$\|\phi\|^{p}_{W^{s,p}_0(\G)} \leq \|u\|^{p}_{W^{s,p}_0(\G)}.$$
  Since $u$ is chosen arbitrarily, we get that $T_{\rho_{0},p,q,s} \geq\|\phi\|^{p}_{W^{s,p}_0(\G)}$. Therefore, we have $T_{\rho_{0},p,q,s}=\|\phi\|^{p}_{W^{s,p}_0(\G)}$ and this show that  $\phi$ is a minimiser of $T_{\rho_{0},p,q,s}$. \end{proof}

Now we are ready to obtain an expression for $C_{GN, \G}$. In other words, we will now present a proof of Theorem \ref{rt1.3}.

\begin{proof}[{\bf Proof of Theorem \ref{rt1.3}}]
  We define the functional, for all $u \in W^{s,p}_0(\G) \backslash \{0\},$
\begin{equation*}
	J(u):=\frac{\left(\int_{\mathbb{G}}\int_{\mathbb{G}}\frac{|u(x)-u(y)|^{p}}{\left|y^{-1} x\right|^{Q+ps}} dxdy\right)^{\frac{Q(q-p)}{sp^2}}\left(\int_{\mathbb{G}}|u(x)|^{p}dx\right)^{\frac{spq-Q(q-p)}{sp^2}}}
		{\int_{\mathbb{G}}|u(x)|^{q}dx}.
\end{equation*}
Then the sharp estimate of $C_{GN, \G}$ can be estimated by studying the following minimisation problem

\begin{equation*}
	C_{GN, \G}^{-1}=\inf \left\{J(u): u \in W_{0}^{s,p}\left(\G\right) \backslash\{0\}\right\} .
\end{equation*}
Let $\phi \in W^{s,p}_0(\G)$ be a least energy solution of \eqref{r1.71}. Then, by definition of $J$ we obtain
$$C_{GN, \G}^{-1} \leq J(\phi).$$ From Lemma \ref{rl3.1} we obtain that

	\begin{align*}
		C_{GN, \G}^{-1} \leq J(\phi)=&\frac{\left(\int_{\mathbb{G}}\int_{\mathbb{G}}\frac{|\phi(x)-\phi(y)|^{p}}{\left|y^{-1} x\right|^{Q+ps}} dxdy\right)^{\frac{Q(q-p)}{sp^2}}\left(\int_{\mathbb{G}}|\phi(x)|^{p}dx\right)^{\frac{spq-Q(q-p)}{sp^2}}}
		{\int_{\mathbb{G}}|\phi(x)|^{q}dx}\\
		=&  \frac{\left(\frac{Q(q-p)}{pqs-Q(q-p)} \int_{\G}|\phi|^{p} dx\right)^{\frac{Q(q-p)}{sp^2}}\left(\int_{\mathbb{G}}|\phi(x)|^{p}dx\right)^{\frac{spq-Q(q-p)}{sp^2}}}
		{\frac{pqs}{pqs-Q(q-p)} \int_{\G}|\phi|^{p} dx}   \\
		=& \frac{pqs-Q(q-p)}{pqs} \left(\frac{Q(q-p)}{pqs-Q(q-p)} \right)^{\frac{Q(q-p)}{sp^2}}\left(\int_{\G}|\phi|^{p} dx\right)^{\frac{q-p}{p}}.
	\end{align*}
Now, we will estimate $C_{GN, \G}^{-1}$ from below. For any $u \in W_{0}^{s,p}\left(\G\right) \backslash\{0\}$, we define $$w(x):=\lambda u\left(D_{\mu}(x)\right),$$
where $\lambda$ and $\mu$ are two positive parameters, which will be determined later. It is easy to see, by direct calculations, that 
\begin{equation}\label{eqs1}
    \int_{\mathbb{G}}\int_{\mathbb{G}}\frac{|w(x)-w(y)|^{p}}{|y^{-1}x|^{Q+ps}}dxdy=\lambda^{p} \mu^{ps-Q} \int_{\mathbb{G}}\int_{\mathbb{G}}\frac{|u(x)-u(y)|^{p}}{|y^{-1}x|^{Q+ps}}dxdy,
\end{equation}
\begin{equation}\label{eqs2}
    \int_{\G}|w(x)|^{p} dx=\lambda^{p} \mu^{-Q} \int_{\G}|u(x)|^{p} dx
\end{equation}
and 
\begin{equation}\label{eqs3}
    \int_{\G}|w(x)|^{q} dx=\lambda^{q} \mu^{-Q} \int_{\G}|u(x)|^{q} dx.
\end{equation}
	
We choose $\lambda$ and $\mu$ such that 
\begin{equation*}
	\lambda^{q} \mu^{-Q} \int_{\G}|u(x)|^{q} dx:=\int_{\G}|\phi(x)|^{q} dx=\frac{pq s}{p qs-Q(q-p)} \int_{\G}|\phi(x)|^{p} dx 
\end{equation*}
and 
\begin{equation*}
	\lambda^{p} \mu^{-Q} \int_{\G}|u(x)|^{p} dx:=\int_{\G}|\phi(x)|^{p} dx.
\end{equation*}

From the above two equalities, it is straightforward to  see that

\begin{equation*}
	\lambda^{p}=\left(\frac{pqs}{ pqs-Q(q-p)}\right)^{\frac{p}{q-p}}\left(\int_{\G}|u|^{p} dx\right)^{\frac{p}{q-p}}\left(\int_{\G}|u|^{q} dx\right)^{-\frac{p}{q-p}}
\end{equation*}

and
\begin{align*}
		\mu^{ps-Q}=&\left(\frac{ pqs}{pqs-Q(q-p)}\right)^{\frac{p}{q-p} \cdot \frac{ps-Q}{Q}}\left(\int_{\G}|u|^{p} dx\right)^{\frac{q}{q-p} \cdot \frac{ps-Q}{Q}} \\
		&\quad\quad\quad \times\left(\int_{\G}|u|^{q} dx\right)^{-\frac{p}{q-p} \cdot \frac{ps-Q}{Q}}\left(\int_{\G}|\phi|^{p} dx\right)^{-\frac{ps-Q}{Q}}.
	\end{align*}

From the choice of the parameters $\lambda$ and $\mu$ as above and from \eqref{eqs2} and \eqref{eqs3},  we have that $$\int_{\G}|w|^{p} dx=\int_{\G}|\phi|^{p} dx$$ and $$\int_{\G}|w|^{q} dx=\int_{\G}|\phi|^{q} dx.$$ 

Since $\phi$ is a minimiser of $T_{\rho_{0},p,q,s}$ with $\rho_{0}=\int_{\G}|\phi|^{q} dx$, by Lemma \ref{rl3.2} we get that
\begin{equation*}
	\int_{\mathbb{G}}\int_{\mathbb{G}}\frac{|w(x)-w(y)|^{p}}{|y^{-1}x|^{Q+ps}}dxdy \geq \int_{\mathbb{G}}\int_{\mathbb{G}}\frac{|\phi(x)-\phi(y)|^{p}}{|y^{-1}x|^{Q+ps}}dxdy.
\end{equation*}

Therefore, using \eqref{eqs1} and by  Lemma \ref{rl3.1} we get
\begin{align*}
    \lambda^{p} \mu^{ps-Q} &\int_{\mathbb{G}}\int_{\mathbb{G}}\frac{|u(x)-u(y)|^{p}}{|y^{-1}x|^{Q+ps}}dxdy=\int_{\mathbb{G}}\int_{\mathbb{G}}\frac{|w(x)-w(y)|^{p}}{|y^{-1}x|^{Q+ps}}dxdy\\&\geq \int_{\mathbb{G}}\int_{\mathbb{G}}\frac{|\phi(x)-\phi(y)|^{p}}{|y^{-1}x|^{Q+ps}}dxdy=\frac{Q(q-p)}{pqs-Q(q-p)} \int_{\G}|\phi|^{p} dx.
\end{align*}

Therefore, by substituting the value of $\lambda^p$ and $\mu^{ps-Q}$  we have

	\begin{align*}
		&\left(\frac{pqs}{ pqs-Q(q-p)}\right)^{\frac{p}{q-p}\left(1+\frac{ps-Q}{Q}\right)}\left(\int_{\G}|u|^{p} dx\right)^{\frac{p}{q-p}+\frac{q}{q-p} \cdot \frac{ps-Q}{Q}}\left(\int_{\G}|u|^{q} dx\right)^{-\frac{p}{q-p}\left(1+\frac{ps-Q}{Q}\right)} \\
		&\quad \times\left(\int_{\G}|\phi|^{p} dx\right)^{-\frac{ps-Q}{Q}} \int_{\mathbb{G}}\int_{\mathbb{G}}\frac{|u(x)-u(y)|^{p}}{|y^{-1}x|^{Q+ps}}dxdy \geq \frac{Q(q-p)}{pqs-Q(q-p)} \int_{\G}|\phi|^{p} dx.
	\end{align*}

It is now deduced by direct calculations that

\begin{equation*}
	J(u) \geq \frac{pqs-Q(q-p)}{ pqs}\left(\frac{Q(q-p)}{pqs-Q(q-p)}\right)^{\frac{Q(q-p)}{p^2s}}\left(\int_{\G}|\phi|^{p} dx\right)^{\frac{q-p}{p}}.
\end{equation*}

Since $u$ is chosen arbitrarily, we get that

\begin{equation*}
	C_{GN, \G}^{-1} \geq \frac{pqs-Q(q-p)}{ pqs}\left(\frac{Q(q-p)}{pqs-Q(q-p)}\right)^{\frac{Q(q-p)}{p^2s}}\left(\int_{\G}|\phi|^{p} dx\right)^{\frac{q-p}{p}}.
\end{equation*}

Therefore, combining the lower and upper estimate of $C_{GN, \G}^{-1}$ we get the first equality of \eqref{r1.10}.

Now, using the property  $d=I(\phi)$ and Lemma \ref{rl3.1} we have
\begin{align*}
    d=I(\phi)&=\frac{1}{p} \iint_{\G \times \G} \frac{|\phi(x)-\phi(y)|^p}{|y^{-1} x|^{Q+ps}} dx dy +\frac{1}{p} \int_{\G} |\phi(x)|^{p} dx -\frac{1}{q} \int_{ \G}|\phi(x)|^{q} dx\\&=\frac{Q(q-p)}{p^2qs-pQ(q-p)} \int_{\G}|\phi|^{p} dx+\frac{1}{p} \int_{\G} |\phi(x)|^{p} dx-\frac{ps}{pqs-Q(q-p)} \int_{\G}|\phi|^{p} dx\\&=\frac{(q-p)s}{pqs-Q(q-p)} \int_{\G}|\phi|^{p} dx.
\end{align*}
This implies that 
\begin{equation} \label{exofd}
	\int_{\G}|\phi|^{p} dx=\frac{pqs-Q(q-p)}{(q-p)s} d,
\end{equation}
which combined with the first equality of \eqref{r1.10} gives 
$$C^{-1}_{GN, \G}=\frac{pqs-Q(q-p)}{pqs}\left(\frac{Q(q-p)}{pqs-Q(q-p)}\right)^{\frac{Q(q-p)}{sp^2}}\left(\frac{pqs-Q(q-p)}{(q-p)s} d\right)^{\frac{q-p}{p}},$$
proving the second equality of \eqref{r1.10}. Hence, the proof of Theorem \ref{rt1.3} is completed. 
\end{proof}

\section{Best contants on the subelliptic fractional Sobolev inequalities }\label{sec5}



The purpose of this section is to discuss  the sharp constant $C_{S, \G}$ in the  subelliptic fractional Sobolev inequalities on the stratified Lie groups.
Here $C_{S, \G}$ is the smallest constant among all positive constants $C$ such that

\begin{equation}\label{r4.2}
	\left(\int_{\mathbb{G}}|u(x)|^{q}dx\right)^{\frac{p}{q}}\leq C\left[\int_{\mathbb{G}}\int_{\mathbb{G}}\frac{|u(x)-u(y)|^{p}}{\left|y^{-1} x\right|^{Q+ps}} dxdy+\int_{\mathbb{G}}|u(x)|^{p})dx\right]
\end{equation}

holds for any $u \in W_{0}^{s,p}\left(\G\right)$. Then $C_{S, \G}$ can be given in the following form:

\begin{equation}\label{r4.3}
	C_{S, \G}^{-1}:=\inf _{u \in W_{0}^{s,p}\left(\G\right) \backslash\{0\}} \frac{\left[\int_{\mathbb{G}}\int_{\mathbb{G}}\frac{|u(x)-u(y)|^{p}}{\left|y^{-1} x\right|^{Q+ps}} dxdy+\int_{\mathbb{G}}|u(x)|^{p})dx\right]}{\left(\int_{\mathbb{G}}|u(x)|^{q}dx\right)^{\frac{p}{q}}}.
\end{equation}

In the following result, we present an expression for $C_{S, \G}$ in terms of the least energy solution of \eqref{r1.71}.

\begin{theorem}\label{rt4.1}
	Let $p<q<p_s^*:=\frac{Qp}{Q-ps}$ and $s \in (0, 1)$. Let $\phi$ be a least energy solution of \eqref{r1.71} and  let  $C_{S, \mathbb{G}}>0$ be the least positive constant such that the inequality \eqref{r4.2} is true. Then we have
		
		\begin{align}\label{Sobo const}
		C_{S,\mathbb{G}}^{-1}&=\left(\frac{spq}{spq-Q(q-p)}\int_{\mathbb{G}}|\phi(x)|^{p}dx\right)^{\frac{q-p}{q}}=\left(\frac{pq}{q-p} d \right)^{\frac{q-p}{q}},
		\end{align}
		where $d=I(\phi):=\inf_{u\in\mathcal{N}}I(u)$. 
\end{theorem}
\begin{proof} Since $\phi$ is a least energy solution of  \eqref{r1.7},  $\mathcal{L}(\phi)=0$ and  Lemma \ref{rl3.1} imply  that
	
	\begin{align*}
	    \frac{\left[\int_{\mathbb{G}}\int_{\mathbb{G}}\frac{|\phi(x)-\phi(y)|^{p}}{\left|y^{-1} x\right|^{Q+ps}} dxdy+\int_{\mathbb{G}}|\phi(x)|^{p})dx\right]}{\left(\int_{\mathbb{G}}|\phi(x)|^{q}dx\right)^{\frac{p}{q}}}&=\left(\int_{\G}|\phi|^{q} dx\right)^{\frac{q-p}{q}}\\&=\left(\frac{sp q}{spq-Q(q-p)} \int_{\G}|\phi|^{p} dx\right)^{\frac{q-p}{q}},
	\end{align*}
	which further shows that 

\begin{equation} \label{sub}
		C_{S, \G}^{-1} \leq\left(\frac{sp q}{spq-Q(q-p)} \int_{\G}|\phi|^{p} dx\right)^{\frac{q-p}{q}}.
	\end{equation}
	
	We define $\tilde{u}(x):=\|\phi\|_{L^{q}(\G)}\|u\|_{L^{q}(\G)}^{-1} u(x)$ for any $u \in W_{0}^{s,p}\left(\G\right) \backslash\{0\}$ so that $$\int_{\G}|\tilde{u}(x)|^{q} dx=\int_{\G}|\phi(x)|^{q} dx.$$ Therefore we have from Lemma \ref{rl3.2} that
	
	\begin{align} \label{lab6.6}
	  \nonumber  \int_{\mathbb{G}}\int_{\mathbb{G}}\frac{|\tilde{u}(x)-\tilde{u}(y)|^{p}}{\left|y^{-1} x\right|^{Q+ps}} dxdy+\int_{\mathbb{G}}|\tilde{u}(x)|^{p})dx &\geq \int_{\mathbb{G}}\int_{\mathbb{G}}\frac{|\phi(x)-\phi(y)|^{p}}{\left|y^{-1} x\right|^{Q+ps}} dxdy+\int_{\mathbb{G}}|\phi(x)|^{p})dx\\&=\frac{sp q}{spq-Q(q-p)} \int_{\G}|\phi|^{p} dx,
	\end{align}
 where we have used \eqref{se2} in the last equality. Now a simple calculation gives that 
\begin{align*}
    &\frac{\left[\int_{\mathbb{G}}\int_{\mathbb{G}}\frac{|u(x)-u(y)|^{p}}{\left|y^{-1} x\right|^{Q+ps}} dxdy+\int_{\mathbb{G}}|u(x)|^{p})dx\right]}{\left(\int_{\mathbb{G}}|u(x)|^{q}dx\right)^{\frac{p}{q}}} \\&\quad\quad\quad\quad= \Bigg(  \int_{\mathbb{G}}\int_{\mathbb{G}}\frac{|\tilde{u}(x)-\tilde{u}(y)|^{p}}{\left|y^{-1} x\right|^{Q+ps}} dxdy+\int_{\mathbb{G}}|\tilde{u}(x)|^{p})dx\Bigg) \|\phi\|_{L^q(\G)}^{-p},
\end{align*}

and therefore, using \eqref{se1} and \eqref{lab6.6}  we obtain that
	\begin{align} \label{slb1}
	    &\frac{\left[\int_{\mathbb{G}}\int_{\mathbb{G}}\frac{|u(x)-u(y)|^{p}}{\left|y^{-1} x\right|^{Q+ps}} dxdy+\int_{\mathbb{G}}|u(x)|^{p})dx\right]}{\left(\int_{\mathbb{G}}|u(x)|^{q}dx\right)^{\frac{p}{q}}} \nonumber \\  \nonumber \quad\quad\quad\quad&\geq \frac{sp q}{spq-Q(q-p)} \left(\int_{\G}|\phi|^{p} dx  \right)\left(\int_{\mathbb{G}}|\phi(x)|^{q}dx\right)^{-\frac{p}{q}}\\&= \left(\frac{sp q}{spq-Q(q-p)}\int_{\mathbb{G}}|\phi(x)|^{p}dx\right)^{\frac{q-p}{q}}.
	\end{align}
	By the arbitrariness of $u \in W^{s,p}_0(\G)$ in \eqref{slb1} we deduce that
	
	\begin{equation}\label{slb}
		C_{S, \G}^{-1} \geq\left(\frac{sp q}{spq-Q(q-p)}\int_{\mathbb{G}}|\phi(x)|^{p}dx\right)^{\frac{q-p}{q}}.
	\end{equation}
	Finally, the estimates \eqref{sub} and \eqref{slb} imply the first equality in \eqref{Sobo const}. The last equality is an immediate consequence of a substitution  \eqref{exofd} in the first equality of \eqref{Sobo const}. This completes the proof. \end{proof}

\section{ Subelliptic fractional  logarithmic Sobolev inequalities with explicit constants}  \label{sec6}
The aim of this section is to obtain the subelliptic  fractional  logarithmic Sobolev inequalities  with explicit constants on the stratified Lie groups. We adopt the method developed by Merker \cite{Mer08} (see also \cite{CKR21b}). In this regard, we first prove the following result extending a result of Merker from the Euclidean space to general measure spaces. 
\begin{theorem}\label{hol+loghol} Let $\mathbb{X}$ be a measure space. 
   The following statements are equivalent: 
   \begin{itemize}
       \item[(i)] The H\"older type inequality 
       \begin{equation} \label{eq4.1}
            \|u\|_{L^r(\mathbb{X})} \leq \|u \|_{L^p(\mathbb{X})}^{a} \|u\|_{L^q(\mathbb{X})}^{1-a}
       \end{equation}
       is valid for all $u \in L^p(\mathbb{X}) \cap L^q(\mathbb{X})$ for the parameters $p,q$ and $r$    ranging over $0<p\leq r\leq q \leq \infty$ and satisfying $\frac{1}{r}=\frac{a}{p}+\frac{1-a}{q}$ with $a \in [0, 1].$
       \item[(ii)] The logarithmic H\"older type inequality 
       \begin{equation} \label{eqlog}
    \int_{\mathbb{X}} \frac{|u(x)|^p}{\|u\|_{L^p(\mathbb{X})}^p} \log \left( \frac{|u(x)|^p}{\|u\|_{L^p(\mathbb{X})}^p} \right) dx\leq \frac{q}{q-p} \log \left(\frac{\|u\|^p_{L^q(\mathbb{X})}}{\|u\|_{L^p(\mathbb{X})}^p} \right),
\end{equation} 
is valid for all $u \in L^p(\mathbb{X}) \cap L^q(\mathbb{X})$ for the parameters $p$ and $q$ satisfying $0<p <q \leq \infty.$
   \end{itemize}
\end{theorem}
\begin{proof} We first observe that the validity of \eqref{eq4.1} for the parameters $p,q$ and $r$    ranging over $0<p\leq r\leq q \leq \infty$ and satisfying $\frac{1}{r}=\frac{a}{p}+\frac{1-a}{q}$ with $a \in [0, 1]$ is equivalent to the convexity of the function $\phi: \frac{1}{r} \mapsto \log(\|u\|_r).$ Indeed, this follows by taking logarithms on both sides of \eqref{eq4.1} to have 
$$\phi\left( \frac{1}{r} \right)\leq a\phi\left( \frac{1}{p} \right)+(1-a) \phi\left( \frac{1}{q} \right),$$
with $a \in [0, 1]$ satisfying  $\frac{1}{r}=\frac{a}{p}+\frac{1-a}{q}.$
Now, the derivative of  the function $$\phi(h):=\log(\|u\|_{\frac{1}{h}})=h \log \Big( \int_{\mathbb{X}} |u(x)|^{\frac{1}{h}} dx \Big)$$
is given by 
$$\phi^{'}(h)= \log\Big( \int_{\mathbb{X}} |u(x)|^{\frac{1}{h}} dx \Big)- \frac{1}{h} \frac{\int_{\mathbb{X}} |u(x)|^{\frac{1}{h}} \log(|u(x)|)dx}{\int_{\mathbb{X}} |u(x)|^{\frac{1}{h}} dx}.$$
Since the  convexity of $\phi$ on $[0, \infty)$ is equivalent to $\phi^{'}(h) \geq \frac{\phi(h_1)-\phi(h)}{h_1-h}$ for $h>h_1 \geq 0,$ by taking $h=\frac{1}{p}$ and $h_1=\frac{1}{q}$ the convexity of $\phi$ is equivalent to 
\begin{equation}
   \log \Big( \int_{\mathbb{X}} |u(x)|^{p} dx \Big)- p \frac{\int_{\mathbb{X}} |u(x)|^{p} \log(|u(x)|)dx}{\int_{\mathbb{X}} |u(x)|^{p} dx} \geq \frac{qp}{p-q} \log \Big(\frac{\|u\|_{L^q(\mathbb{X})}}{\|u\|_{L^p(\mathbb{X})}} \Big)
\end{equation}
and therefore, equivalent to
\begin{align}
   \nonumber &\int_{\mathbb{X}} \frac{|u(x)|^p}{\|u\|_{L^p(\mathbb{X})}^p} \log \left( \frac{|u(x)|^p}{\|u\|_{L^p(\mathbb{X})}^p} \right) dx \\&=  p \frac{\int_{\mathbb{X}} |u(x)|^{p} \log(|u(x)|)dx}{\int_{\mathbb{X}} |u(x)|^{p} dx} -\log \Big( \int_{\mathbb{X}} |u(x)|^{p} dx \Big) \leq \frac{q}{q-p} \log \left(\frac{\|u\|^p_{L^q(\mathbb{X})}}{\|u\|_{L^p(\mathbb{X})}^p} \right),
\end{align} establishing the validity of \eqref{eqlog} for the range $0<p<q \leq \infty$ and its equivalence with \eqref{eq4.1} as this process is reversible at every step.
\end{proof}

The following result follows from Theorem \ref{hol+loghol} by noting that the inequality \eqref{eq4.1} holds for all $1\leq p\leq r \leq q \leq \infty$ satisfying $\frac{1}{r}=\frac{a}{q}+\frac{1-a}{q}$ with $a \in [0, 1]$ as a result of the usual H\"older inequality in a measure space. We note here that this result is an  improvement of  a result established by the third author and collaborators \cite{CKR21b} in view of the inclusion of  endpoints.

\begin{theorem}\label{LogHol}
Let $\mathbb{X}$ be a measure space. Let $u \in L^p(\mathbb{X}) \cap L^q(\mathbb{X}) \backslash \{0\}$ with $1\leq  p<q\leq \infty.$ Then we have  
\begin{equation} \label{LogHOleq}
    \int_{\mathbb{X}} \frac{|u(x)|^p}{\|u\|_{L^p(\mathbb{X})}^p} \log \left( \frac{|u(x)|^p}{\|u\|^p_{L^p(\mathbb{X})}} \right) dx\leq \frac{q}{q-p} \log \left(\frac{\|u\|^p_{L^q(\mathbb{X})}}{\|u\|_{L^p(\mathbb{X})}^p} \right).
\end{equation}
\end{theorem}

Next result is the subelliptic fractional logarithmic Sobolev inequality for inhomogeneous fractional Sobolev spaces on stratified Lie groups.

\begin{theorem} Let $s \in (0, 1),$ $p \in (1, \infty), Q>ps,$ and $p<q<p^*:=\frac{Qp}{Q-ps},$ where $Q$ is the homogeneous dimension of a stratified Lie group $\mathbb{G}.$
 Then for any $u \neq 0$ we have 
\begin{align} \label{logsobo}
    \int_{\G} \frac{|u(x)|^p}{\|u\|_{L^p(\G)}^p} \log \left( \frac{|u(x)|^p}{\|u\|^p_{L^p(\G)}} \right) dx \leq \frac{Q}{s} \log \left( C_{S, \G, p}\frac{\left[\int_{\mathbb{G}}\int_{\mathbb{G}}\frac{|u(x)-u(y)|^{p}}{\left|y^{-1} x\right|^{Q+ps}} dxdy+\int_{\mathbb{G}}|u(x)|^{p})dx\right]}{\|u\|_{L^p(G)}^p} \right),
\end{align} where $C_{S, \G, p}$ is  given by 
$$C_{S, \G, p}:=\left( \frac{s}{Qd}\right)^{\frac{sp}{Q}}.$$
\end{theorem}
Here $d=I(\phi):=\inf_{u\in\mathcal{N}}I(u)$, $\phi$ being a least energy solution of \eqref{r1.71}.

\begin{proof}
    The logarithmic H\"older inequality \eqref{LogHol} in the case when  $\mathbb{X}=(\G, dx)$  and the fractional subelliptic Sobolev inequality \eqref{ineq-Sobolev} imply that
\begin{align} \label{eq7.7}
   \nonumber  \int_{\G} \frac{|u(x)|^p}{\|u\|_{L^p(\G)}^p}& \log \left( \frac{|u(x)|^p}{\|u\|^p_{L^p(\G)}} \right) dx \leq \frac{q}{q-p}\log \left(\frac{\|u\|^p_{L^q(\G)}}{\|u\|_{L^p(\G)}^p} \right)\\& \leq \frac{q}{q-p}\log \left(C_{S, \G} \frac{\left[\int_{\mathbb{G}}\int_{\mathbb{G}}\frac{|u(x)-u(y)|^{p}}{\left|y^{-1} x\right|^{Q+ps}} dxdy+\int_{\mathbb{G}}|u(x)|^{p})dx\right]}{\|u\|_{L^p(\G)}^p} \right).
\end{align}
Observe that the only place where the parameter $q$ appears in \eqref{eq7.7} are  constants $\frac{q}{q-p}$ and $C_{S, \G}.$ Therefore, by minimising the constant $\frac{q}{q-p}$ over the range of $q \in (p, p_s^*)$ we get that 
$\underset{{q \in (p, p_s^*)}}{\min} \frac{q}{q-p}= \frac{\frac{Qp}{Q-ps}}{\frac{Qp}{Q-ps}-p}= \frac{Q}{sp}.$ Therefore, from \eqref{eq7.7} we deduce that 
\begin{equation*}
    \int_{\G} \frac{|u(x)|^p}{\|u\|_{L^p(\G)}^p} \leq \frac{Q}{sp}\log \left(C_{S, \G, p} \frac{\left[\int_{\mathbb{G}}\int_{\mathbb{G}}\frac{|u(x)-u(y)|^{p}}{\left|y^{-1} x\right|^{Q+ps}} dxdy+\int_{\mathbb{G}}|u(x)|^{p})dx\right]}{\|u\|_{L^p(\G)}^p} \right) ,
\end{equation*}
where the constant $C_{S, \G, p}$ is given by 
$$C_{S, \G, p}:=\limsup_{q \rightarrow \frac{Qp}{Q-ps}}C_{S, \G}=\limsup_{q \rightarrow \frac{Qp}{Q-ps}}\left(\frac{pq}{q-p} d \right)^{\frac{p-q}{q}}=\left( \frac{s}{Qd}\right)^{\frac{sp}{Q}},$$ completing the proof.
\end{proof}
The following theorem gives the subelliptic fractional logarithmic Sobolev inequality for homogeneous fractional Sobolev spaces on stratified Lie groups.
\begin{theorem} Let $s \in (0, 1),$ $p \in (1, \infty), Q>ps,$ and $p<q<p^*:=\frac{Qp}{Q-ps},$ where $Q$ is the homogeneous dimension of a stratified Lie group $\mathbb{G}.$
 Then for any $u \neq 0$, we have 
\begin{align} \label{logsobodunk}
    \int_{\G} \frac{|u(x)|^p}{\|u\|_{L^p(\G)}^p} \log \left( \frac{|u(x)|^p}{\|u\|^p_{L^p(\G)}} \right) dx \leq \frac{Q}{s} \log \left( C_{GN, \G}^{\frac{ps}{Q(q-p)}}\frac{\left(\int_{\mathbb{G}}\int_{\mathbb{G}}\frac{|u(x)-u(y)|^{p}}{\left|y^{-1} x\right|^{Q+ps}} dxdy\right)^{\frac{1}{p}}}{\|u\|_{L^p(G)}} \right),
\end{align} where $C_{GN, \G}$ is the best constant in the fractional subelliptic Gagliardo-Nirenberg inequality given by 

$$C_{GN, \G}:=\frac{pqs}{pqs-Q(q-p)}\left(\frac{Q(q-p)}{pqs-Q(q-p)}\right)^{\frac{Q(p-q)}{sp^2}}\left(\frac{pqs-Q(q-p)}{(q-p)s} d\right)^{\frac{p-q}{p}}.$$
\end{theorem}
Here $d=I(\phi):=\inf_{u\in\mathcal{N}}I(u)$, $\phi$ being a least energy solution of \eqref{r1.71}.
\begin{proof}
    Using the fractional subelliptic  Gagliardo-Nirenberg inequality \eqref{r1.6} and the logarithmic H\"older inequality \eqref{LogHOleq}, we obtain that 
    \begin{align*}
        \int_{\G} \frac{|u(x)|^p}{\|u\|_{L^p(\G)}^p}& \log \left( \frac{|u(x)|^p}{\|u\|^p_{L^p(\G)}} \right) dx \leq \frac{q}{q-p}\log \left(\frac{\|u\|^p_{L^q(\G)}}{\|u\|_{L^p(\G)}^p} \right) \\& \leq \frac{q}{q-p} \log \left( C_{GN, \G}^{\frac{p}{q}}\frac{\left(\int_{\mathbb{G}}\int_{\mathbb{G}}\frac{|u(x)-u(y)|^{p}}{\left|y^{-1} x\right|^{Q+ps}} dxdy\right)^{\frac{Q(q-p)}{spq}}\left(\int_{ \mathbb{G}}|u|^{p} dx\right)^{\frac{spq-Q(q-p)}{spq}}}{\|u\|^p_{L^p(G)}} \right) \\& =\frac{q}{q-p} \log \left( C_{GN, \G}^{\frac{p}{q}}\frac{\left(\int_{\mathbb{G}}\int_{\mathbb{G}}\frac{|u(x)-u(y)|^{p}}{\left|y^{-1} x\right|^{Q+ps}} dxdy\right)^{\frac{Q(q-p)}{spq}}}{\|u\|^{\frac{Q(q-p)}{sq}}_{L^p(G)}} \right)\\&=
        \frac{Q}{s} \log \left( C_{GN, \G}^{\frac{ps}{Q(q-p)}}\frac{\left(\int_{\mathbb{G}}\int_{\mathbb{G}}\frac{|u(x)-u(y)|^{p}}{\left|y^{-1} x\right|^{Q+ps}} dxdy\right)^{\frac{1}{p}}}{\|u\|_{L^p(G)}} \right),
    \end{align*}
    completing the proof of \eqref{logsobodunk}.
\end{proof}

\section{Conflict of interest statement}
On behalf of all authors, the corresponding author states that there is no conflict of interest.

\section{Data availability statement}
Data sharing is not applicable to this article as no datasets were generated or analysed during the current study.

\section{Acknowledgement}
   VK and MR are supported by the FWO Odysseus 1 grant G.0H94.18N: Analysis and Partial
Differential Equations, the Methusalem programme of the Ghent University Special Research Fund (BOF) (Grant number 01M01021) and by FWO Senior Research Grant G011522N. MR is also supported by EPSRC grant
EP/R003025/2.

\end{document}